\documentclass[12pt]{article}%
   
\usepackage{amsmath,enumerate}
\usepackage{amsfonts}
\usepackage{amssymb}
\usepackage{comment}
\usepackage{color}

\newtheorem{theorem}{Theorem}[section]

\newtheorem{claim}[theorem]{Claim}

\newtheorem{lemma}[theorem]{Lemma}

\newtheorem{problem}[theorem]{Problem}
\newtheorem{proposition}[theorem]{Proposition}

\newenvironment{proof}[1][Proof]{\noindent\textbf{#1.} }
{\hfill \ \rule{0.5em}{0.5em}}
\newcommand{\sat}{\textup{sat}}

\title{Few $T$ copies in $H$-saturated graphs}
\author{J\"urgen Kritschgau\thanks{Department of Mathematics, Iowa State University, \texttt{jkritsch@iastate.edu}} \and Abhishek Methuku \thanks{Central European University, Budapest and \'Ecole Polytechnique F\'ed\'erale de Lausanne, \texttt{abhishekmethuku@gmail.com}. Research is supported by the National Research, Development and Innovation Office - NKFIH under the grant K 116769.} \and Michael Tait\thanks{Department of Mathematical Sciences, Carnegie Mellon University, \texttt{mtait@cmu.edu}. Research is supported in part by NSF grant DMS-1606350.} \and Craig Timmons \thanks{Department of Mathematics and Statistics, California State University, Sacramento.
Research is supported by a grant from the Simons Foundation \#359419.}}

\begin{document}

\maketitle
\begin{abstract}
A graph is $F$-saturated if it is $F$-free but the addition of any edge creates a copy of $F$. In this paper we study the quantity $\mathrm{sat}(n, H, F)$ which denotes the minimum number of copies of $H$ that an $F$-saturated graph on $n$ vertices may contain. This parameter is a natural saturation analogue of Alon and Shikhelman's generalized Tur\'an problem, and letting $H = K_2$ recovers the well-studied saturation function. We provide a first investigation into this general function focusing on the cases where the host graph is either $K_s$ or $C_k$-saturated. Some representative interesting behavior is:

\begin{enumerate}[(a)]
\item For any natural number $m$, there are graphs $H$ and $F$ such that $\sat(n, H, F) = \Theta(n^m)$.
\item For many pairs $k$ and $l$, we show $\mathrm{sat}(n, C_l, C_k) = 0$. In particular, we prove that there exists a triangle-free $C_k$-saturated graph on $n$ vertices for any $k > 4$ and large enough $n$. 
\item $\mathrm{sat}(n, K_3, K_4) = n-2$, $\mathrm{sat}(n, C_4, K_4) \sim \frac{n^2}{2}$, and $\mathrm{sat}(n, C_6, K_5) \sim n^3$. 

\end{enumerate}

We discuss several intriguing problems which 
remain unsolved.
\end{abstract}

\section{Introduction}

Given graphs $G$ and $F$, the graph $G$ is \emph{$F$-free} if $G$ does not contain $F$ as a 
subgraph.  We write $\textup{ex}(n , F)$
for the \emph{Tur\'{a}n number} of $F$ which is the maximum number of edges in an 
$n$-vertex $F$-free graph. This function is a fundamental object in combinatorics, c.f. \cite{sidorenko} for a survey.  An important generalization of the Tur\'{a}n number 
was introduced by Alon and Shikhelman \cite{as}.  For a graph $H$, 
write 
\[
\textup{ex}(n , H , F)
\]
 for the maximum number of copies of $H$ in an $F$-free $n$-vertex graph.
 Taking $H = K_2$ gives back the ordinary Tur\'{a}n number 
 $\textup{ex}(n , F)$.  
 The function $\textup{ex}(n , H , F)$ 
 has been studied by numerous researchers
 (\cite{alonks, egms, GGMV, gerbnerpalmer, GS, maqiu, samotij} to name a few).
 
 Next we discuss graph saturation.  
 The graph $G$ is \emph{$F$-saturated} if $G$ is $F$-free, but adding any nonedge 
to $G$ creates at least one copy of $F$. 
The \emph{saturation number} of $F$, written $
\textup{sat}(n , F)$,
is the minimum number of edges in an $n$-vertex $F$-saturated graph. Saturation in graphs has been studied extensively since the 1960s, c.f. \cite{sat survey} for a survey.  
Generalizing the function $\textup{sat}(n ,F)$, we write 
\[
\textup{sat}(n , H  , F)
\]
for the minimum number of copies of $H$ in an $n$-vertex $F$-saturated graph.    
Note that 
\[
\textup{sat}(n , K_2 , F) = \textup{sat}(n  , F).
\]
If $F$ is a subgraph of $H$, then $\textup{sat}(n , H  , F) = 0$.
This is because an $n$-vertex $F$-free graph with $\textup{ex}(n , F)$ edges 
is $F$-saturated, and has no copies of $H$.  

One of the first results in this area is a theorem 
of Erd\H{o}s, Hajnal, and Moon \cite{ehm} which determines
the saturation number of any complete graph.

\begin{theorem}[Erd\H{o}s, Hajnal, Moon]\label{theorem ehm}
If $s \geq 3$ is an integer, then
\[
\textup{sat}(n , K_s) = (s-2) ( n - s +2) + \binom{s-2}{2}.
\]
Furthermore, if $G$ is an $n$-vertex $K_s$-saturated
graph with $\textup{sat}(n , K_s)$ edges,
then $G$ is isomorphic to the join of a clique 
with $s-2$ vertices and an independent set with $n-s+2$ vertices.
\end{theorem}

While Theorem \ref{theorem ehm} solves the saturation problem
for complete graphs, many other cases have 
since been studied, including cycles.  

Our main results fall into two categories: counting graphs in $K_s$-saturated graphs or counting graphs in $C_k$-saturated graphs.
We first discuss counting graphs in $K_s$-saturated graphs 
as this line of research is a natural 
generalization of Theorem \ref{theorem ehm} in 
the spirit of Alon and Shikhelman \cite{as}.

\subsection{Clique saturated graphs}

\begin{theorem}\label{small improvement}
Let $s > r \geq 3$ be integers.  There is a constant $n_{s,r}$ such that for all $n \geq n_{s,r}$, 
\begin{eqnarray*}
\max \left\{
\frac{ \binom{s-2}{r-1} }{r-1} n - 2 \binom{s-2}{r-1}  , 
\left( \frac{ \binom{s-2}{r-1} + \binom{s-3}{r-2} }{r} \right) n   \right\}
& \leq & \textup{sat}(n , K_r , K_s ) \\
& \leq & (n - s + 2) \binom{s - 2}{r-1} + \binom{s-2}{r}.
\end{eqnarray*}
\end{theorem}

\bigskip

The 
join of a clique with $s-2$ vertices 
and an independent set with $n - s +2$ vertices gives 
the upper bound of Theorem \ref{small improvement}.  
This is the same graph that is the unique extremal 
example for Theorem \ref{theorem ehm}.
For $r \geq \sqrt{s-1} + 1$, the second entry in the maximum 
gives the better lower bound.  When $r$ is fixed and $s$ tends to 
infinity, the lower bound is roughly $\frac{1}{r-1} \binom{s-2}{r-1}n$ which means 
there is a gap of a factor of $r-1$ between the lower and upper bounds.

Theorem \ref{small improvement} shows that $\textup{sat}( n , K_r , K_s) = \Theta(n)$ for $n \geq s > r \geq 3$, 
but it does not give an asymptotic formula.  
When $s = 4$ and $r = 3$, Theorem \ref{small improvement} implies
  \[
  \frac{2n}{3} \leq \textup{sat}(n , K_3 , K_4 ) \leq n - 2.
  \]
 In this special case we can determine $\textup{sat}(n , K_3 , K_4)$ exactly.  
  
\begin{theorem}
\label{K4triangletheorem}
For $n \geq 7$, 
\[
\textup{sat}(n , K_3 , K_4) = n -2.
\]
Furthermore, 
the only $n$-vertex $K_4$-saturated graph with $n - 2$ triangles is 
the join of an edge and an independent set with $n-2$ vertices.
\end{theorem}

K\'{a}szonyi and Tuza \cite{kt} proved that for any graph $F$, there 
is a constant $C$, depending only on $F$, for which 
\begin{equation}\label{kt inequality}
\textup{sat}(n , F) < C n.
\end{equation}
We can use the same construction that gives (\ref{kt inequality}) to 
prove that $\textup{sat}(n , K_r , F)$ is also at most $Cn$.

\begin{proposition}\label{general ub}
Let $n \geq 1$ and $r \geq 2$ be integers.  For any graph $F$,  
there is a constant $C = C(r,F)$ such that 
\[
\textup{sat}(n , K_r , F ) < Cn.
\]
\end{proposition}

If one replaces $K_r$ with an arbitrary graph $H$ in Proposition \ref{general ub}, 
then it is not necessarily the case that $\textup{sat}(n , H , F) = O(n)$, as the following result shows.

\vspace{2mm}

Let $H$ be a noncomplete graph with at most $s$ vertices.  Fix 
some nonedge $h_1 h_2$ in $H$, and let $u_1$ and $u_2$ be a 
fixed pair of vertices in $K_s$.
Let $f_{h_1 , h_2 } (H)$ be the number of copies of $H$ in $K_s$ 
where the vertices $h_1$, $h_2$ of $H$ correspond to the vertices $u_1$, $u_2$, respectively, in the $K_s$.


 \begin{proposition}\label{quadratic lower bound}
 Let $s \geq 3$ be an integer. 
If $H$ contains at most $s$ vertices and $h_1 h_2$ is any nonedge of $H$, then 
\[
\textup{sat}(n , H , K_s) \geq f_{h_1 , h_2} (H) \left( \frac{ n^2 }{ 2 ( s -  1) } - \frac{n}{2} \right).
\] 
\end{proposition}


\bigskip

Applying Proposition \ref{quadratic lower bound} with $H = C_4$, we can prove that $\textup{sat}( n , C_4 , K_4 ) \sim \frac{n^2}{2}$. More precisely, we have:

\begin{proposition}\label{c4 against k4}
Let $\delta>0$ be a real number.  There is an $n( \delta )$ such that for all $n \geq n (\delta )$, 
\[
\binom{n}{2} - n^{5/3 + \delta } \leq \textup{sat}( n , C_4 , K_4 ) 
\leq \binom{n-2}{2}.
\] 
\end{proposition}

\bigskip

The graph obtained by taking a vertex of degree 
$n-1$ and putting $\lfloor \frac{n-1}{2} \rfloor$ 
disjoint edges in its neighborhood is 
$(K_4 - e)$-saturated and has no $C_4$, so 
\[
\textup{sat}( n , C_4 , K_4 - e ) = 0.
\]
Thus, even though $K_4 - e$ differs from 
$K_4$ by only one edge, the functions 
$\textup{sat}( n , C_4 , K_4 - e ) $
and 
$\textup{sat}( n , C_4 , K_4 )$ have very different 
behavior.

Proposition \ref{c4 against k4} and Theorem \ref{small improvement} 
gives examples of graphs $H$ and $F$ where 
\[
\textup{sat}(n,H,F) = \Theta (n^m)
\]
for
$m=1,2$.  In fact, for any integer $m   \geq 3$, there
are graphs $H$ and $F$ for which 
$\textup{sat}(n , H, F) = \Theta (n^m)$, as the following result shows.

\begin{theorem}\label{cr against ks with constants}
For $s\geq 5$ and $r \le 2s-4$, $$\sat(n,C_r,K_s)=\Theta (n^{\lfloor \frac{r}{2}\rfloor}).$$ More precisely,

$$\begin{cases}
\left(\frac{(s-2)_k}{4\cdot k}\right)\left(n^k-o(n^k)\right)\leq\sat(n,C_r,K_s)\leq \left( \frac{(s-2)_k}{2k}\right)\left(n^{ k}+o(n^{ k})\right) & \text{if }  2|r\\
\left(\frac{(s-2)_{k+1}(k-2)!}{r(r-3)(r)_k(s-1)}\right)\left(n^k-o(n^{k})\right)\leq\sat(n,C_r,K_s)\leq \left( \frac{(s-2)_{k+1}}{2}\right)\left(n^{ k}+o(n^{ k})\right) & \text{if }  2\not |r
\end{cases}$$ where $k=\lfloor \frac{r}{2}\rfloor$,
and $(m)_k = m (m-1) \cdots (m-k+1)$.
\end{theorem}

In the special case of $\textup{sat}(n , C_6 , K_5)$, Theorem  
\ref{cr against ks with constants} shows
$\textup{sat}(n , C_6, K_5) = \Theta (n^3)$.  With a more 
specialized argument, we can determine this function asymptotically.

\begin{theorem}\label{c6 against k5}
We have 
\[
(1 - o(1)) n^3 \leq \textup{sat}(n , C_6 , K_5) \leq 6 \binom{n-3}{3}.
\]
\end{theorem}


\subsection{Cycle saturated graphs}

Thus far, many of the results we have stated on 
$\textup{sat}(n , H, F)$ concern the cases when $F$ is a complete 
graph, and when $H$ is a cycle or a complete graph.  
The case when $H$ and $F$ are both cycles also 
is interesting.  
Some cases are fairly straightforward.  
A complete bipartite graph with large enough part sizes is $C_{2k+1}$-saturated and $C_{2t+1}$-free.  
Thus,
\[
\textup{sat}(n , C_{2t +1} , C_{2k+1} ) = 0
\]
for all $t , k \geq 1$ and $n \geq k+1$.  


  


This shows that minimizing the number of triangles in 
a $C_{k}$-saturated graph when $k$ is odd is trivial.  

Our next theorem shows that there exist $n$-vertex triangle-free graphs that are $C_k$-saturated for any even $k \ge 5$ when $n$ is large enough.    

\begin{theorem}\label{main theorem k3 against c2k}
For any integer $k \geq 5$, 
\[
\textup{sat}(n , K_3 , C_k ) = 0
\]
for all $n \geq 2k + 2$.  
\end{theorem}

The case $\textup{sat}(n , K_3 , C_4)$ is not covered by Theorem \ref{main theorem k3 against c2k} and 
appears difficult.  This is discussed further in Section \ref{k3s in c4-sat}.  We know that 
$\textup{sat}(n , K_3 , C_4) = 0$ 
for $8 \leq n \leq 24$.    

\vspace{2mm}

In the case of $\textup{sat}(n , C_4 , C_k )$, we have the following result.  The method 
of the proof used for Theorem \ref{main theorem c4 against ck} is very different from the proof of Theorem \ref{main theorem k3 against c2k}.

\begin{theorem}\label{main theorem c4 against ck}
For all $n \geq 111$ and $k \in \{7,8,9,10 \}$, 
\[
\textup{sat}(n , C_4 , C_k ) = 0.
\]
\end{theorem}

The lower bound on $n$ is not needed when $k \in \{7,8 \}$.  For these two cases, 
we have $\textup{sat}(n , C_4 , C_7) = 0$ for all $n \geq 8$, and 
$\textup{sat}(n , C_4 , C_8)=0$ for all $n \geq 9$.  It is likely that the lower bound on $n$ is a consequence of our proof technique and is not 
optimal for $k \in \{9, 10 \}$.   

\vspace{2mm}
Finally, simple bounds on $\textup{sat}(n , C_{2l} , C_{2k} )$ can obtained by the following construction. Let $C$ a clique on $2k - 2$ vertices and fix two vertices $x, z \in C$, and let $y_1, y_2, \ldots, y_{n-2k+2}$ be the vertices not in $C$. Let us add all the edges $xy_i$, $y_iz$ for $1 \le i \le n-2k+2$. This shows that
\[
\textup{sat}(n , C_{2l} , C_{2k} ) 
\leq 
\left\{
\begin{array}{cc}
0 & \mbox{if $l \geq k$,} \\
O_{k,l}(n) & \mbox{if $l < k$.}
\end{array}
\right.
\]

The table shown in Figure \ref{summary of results} gives a summary of our results with references. 
\begin{figure}\label{summary of results}
\begin{tabular}{l|l|l}
 Result& Hypothesis & Reference\\\hline
 $\textup{sat}(n,K_r,F) = O(n)$&$ n\geq 1,r\geq 2$& Proposition \ref{general ub} \\
  $\textup{sat}(n,H,K_s) = \Omega(n^2) $&$ H\neq K_s , |V(H)| = s\geq 3$&Proposition \ref{quadratic lower bound}\\
 $\textup{sat}(n,K_r,K_s)=\Theta(n)$&$ s>r\geq 3$&Theorem \ref{small improvement} \\
 $\textup{sat}(n,K_3,K_4)=n-2$&$ n\geq 7$&Theorem \ref{K4triangletheorem}\\
 $\textup{sat}(n,C_4,K_4) \sim \frac{n^2}{2}$& ~~~ &  Proposition \ref{c4 against k4}  \\
 $\textup{sat}(n,C_r,K_s)=\Theta(n^{\lfloor \frac{r}{2}\rfloor})$ & $s\geq 5, r\leq 2s-4$ & Theorem \ref{cr against ks with constants}\\
   $\textup{sat}(n,C_6,K_5) \sim n^3 $& ~~~ & Theorem \ref{c6 against k5} \\
 $\textup{sat}(n,K_3,C_k)=0$&$ k\geq 5, n\geq 2k+2$& Theorem \ref{main theorem k3 against c2k} \\
 $\textup{sat}(n, C_4, C_k)=0$&$ n\geq 111, k\in \{7,8,9,10\}$&Theorem \ref{main theorem c4 against ck}\\
 $\textup{sat}(n,C_{2l}, C_{2k})=0$&$ l\geq k$&\\
  $\textup{sat}(n,C_{2l}, C_{2k})=O_{k,l}(n)$&$ l< k$&\\
  $\textup{sat}(n,H,C_{2k+1})=0$&$ n\geq 2k+2\geq 4, H\text{ is not bipartite}$&Proposition \ref{H against odd cycle}\\
  $\textup{sat}(n, C_t,C_k)=0$&$ n\geq t\geq k\geq 3$&Proposition \ref{long against short}\\
  $\textup{sat}(m(r-1)+1,C_t,C_k)=0$&$ t\geq r+1, 2r-2\geq k\geq r+1$&Proposition \ref{long against short}\\
  $\textup{sat}(n,K_3,C_4)\leq \lfloor \frac{n-1}{2}\rfloor$&$n\geq 4$&\\
  $\textup{sat}(10t+1,C_4,C_6)\leq 2t$&$ t\geq 1$&Theorem \ref{c4 against c6}
 \end{tabular}
 \center{Figure \ref{summary of results}: Summary of our results}
 \end{figure}

\vspace{2mm}

The rest of this paper is organized as follows.  
In the next section, we introduce some of our notation and 
give the proofs of Propositions \ref{general ub}, \ref{quadratic lower bound}, 
and \ref{c4 against k4}.  
Many of the ideas used in the proofs of these propositions will be 
used at other places in the paper.  Section \ref{section clique saturated} considers $K_s$-saturated graphs and  
contains the proofs of Theorems \ref{small improvement}, \ref{K4triangletheorem}, \ref{cr against ks with constants}, and \ref{c6 against k5}.  Section \ref{section cycles against cycles}
contains our results on $\textup{sat}(n ,H,F)$ where $H$ and $F$ are both
cycles.  The proofs of Theorem \ref{main theorem k3 against c2k} and
\ref{main theorem c4 against ck} are given there along with more discussion. 
We end with some open problems in Section \ref{section problems}.


\section{Notation and Proofs of Propositions \ref{general ub}, 
\ref{quadratic lower bound}, and \ref{c4 against k4}}

Throughout the paper, we write $(n)_k$ to denote the falling factorial $(n)(n-1)\ldots(n-k+1)$.

For two graphs $G$ and $F$, the \emph{join} of $G$ and $F$ is written $G + F$.  
This is the graph obtained by taking the union of $G$ and $F$, and joining every vertex 
of $G$ to every vertex of $F$.  For $r \geq 3$, $\overline{K_r}$ is the graph with 
$r$ vertices and no edges.  
Write $N(v)$ for the neighborhood of $v$, and 
$N_2(v)$ for the vertices at distance 2 from $v$.  
A very useful fact 
about a $K_s$-saturated graph is that its diameter is 2 provided $s \geq 3$.  
Thus, if $v$ is any vertex in a $K_s$-saturated graph $G$, then 
\[
V(G) = \{v \} \cup N(v) \cup N_2 (v).  
\]
We will slightly abuse notation and use $N(v)$ and $N_2(v)$ to represent the subgraph of $G$ induced by 
$N(v)$ and $N_2(v)$, respectively, when it is convenient.    
One observation that we will use frequently is that in a $K_s$-saturated graph, if $u$ and $v$ are two 
vertices that are not adjacent, then $N(u) \cap N(v)$ must contain a copy of $K_{s - 2}$. 
This also implies that the minimum degree in a $K_s$-saturated graph is at least $s-2$.  

Let us finish this section by giving the proofs of Propositions 
\ref{general ub}, \ref{quadratic lower bound}, and \ref{c4 against k4}.  

\bigskip

\begin{proof}[Proof of Proposition \ref{general ub}]
Let $r \geq 2 $ be an integer and $F$ be a graph.  
We use the construction of K\'{a}szonyi and Tuza \cite{kt}
to build an $n$-vertex graph $G$ that is $F$-saturated.
Start with a clique $C$ with $|V(F)| - \alpha (F) - 1$ vertices, and join all vertices of 
$C$ to an independent set $J$ with $n - ( |V(F)| - \alpha (F) - 1)$ vertices.  
The graph constructed so far is $F$-free.  This is because 
all of the edges in this graph can be covered with $|V(F)| - \alpha (F) - 1 = \beta (F) - 1$ vertices 
(here $\beta (F)$ is the minimum number of vertices in $F$ needed to touch all edges of $F$).     
We now add edges to the independent set $J$ one by one until we obtain an $F$-saturated graph $G$.  
We will never add more than $\alpha (F)$ edges incident to 
a single vertex in $J$ because this would create a copy of $F$.  Indeed, if
$v \in J$ and $v$ has $\alpha (F)$ other neighbors in $J$, then the clique 
$C$ together with $v$ forms a clique of size $|V(F)| - \alpha (F)$, and all of these vertices are joined to $\alpha (F)$ neighbors of $v$ in $J$.  This subgraph contains a copy of $F$.  Thus, at the end of the process the subgraph induced by $J$ has maximum degree less than $\alpha (F)$, and all vertices in $J$ have degree 
at most $|V(F)| - \alpha (F) - 1 + \alpha (F) - 1  = |V(F)| - 2$.  

The last step is to estimate 
the number of $K_r$'s in $G$.  There are 
$\binom{ |V(F)| - \alpha (F) - 1 } {r}$ copies of $K_r$ that do not contain a vertex in $J$.
The number of $K_r$'s that contain at least one vertex in $J$ is at most 
\[
n \binom{ |V(F) | - 2 }{ r - 1}
\]
since the degree of a vertex in $J$ is no more than $|V(F) | - 2$.  
We conclude that there are at most 
\[
n \binom{ |V(F)|  -  2}{r-1} + \binom{ |V(F)| - \alpha (F) - 1 } {r}
\]
copies of $K_r$ in the $F$-saturated graph $G$.
\end{proof}

\bigskip
\begin{proof}[Proof of Proposition \ref{quadratic lower bound}]
Let $G$ be an $n$-vertex $K_s$-saturated graph.
So if $xy$ is a nonedge of $G$, then there is an $(s-2)$-clique, 
say on $\{z_1 , \dots , z_{s-2} \}$, where each $z_i$ 
is adjacent to both $x$ and $y$.  There are 
$f_{h_1, h_2} (H)$ copies of $H$ in $G$ which are contained in the
vertex set $\{ x , y , z_1 , \dots , z_{s - 2 } \}$, such that the vertices
$h_1$, $h_2$ of each copy of $H$ correspond to $x$, $y$ respectively.  
Let $H_1 , \dots , H_f$ be these copies of 
$H$ where $f = f_{h_1,  h_2} (H)$.  We claim that 
each time we choose a nonedge and obtain the corresponding 
$f$ many copies of $H$, we will never see the same copy of $H$ twice.
To see this, suppose $xy$ and $x' y'$ are distinct nonedges of $G$.
Let $H_1 , \dots , H_f$ and $H_1 ' , \dots , H_f '$ be copies of 
$H$ obtained from $xy$ and $x' y'$, respectively.  The vertex set 
of each $H_i$ is $\{x,y, z_1 , \dots , z_{s-2} \}$ where 
$\{z_1 , \dots ,z_{s-2} \}$ is an $(s-2)$-clique, and 
$x$ and $y$ are joined to every $z_i$.  If some $H_j '$ has 
the same vertex set as $H_i$, then 
$x' , y ' \in \{x , y , z_1 , \dots , z_{s-2}\}$.  
This is a contradiction since $x' y'$ is a nonedge and the 
only missing edge from $\{x , y , z_1, \dots , z_{s-2} \}$ is 
$xy$.  This shows that the number of copies of $H$ in $G$ is at least 
\[
f_{h_1, h_2 }(H) \left( \binom{n}{2} - e(G) \right) \geq 
f_{h_1, h_2} (H) \left( \binom{n}{2} - \textup{ex}(n, K_s ) \right).
\]  
The proposition follows from the bound 
$\textup{ex}(n , K_s ) \leq \left( 1 - \frac{1}{s - 1} \right) \frac{n^2}{2}$. Observe that since $H$ has a nonedge, $f_{h_1, h_2} (H) \geq 1$.   
\end{proof}

\bigskip

\begin{proof}[Proof of Proposition \ref{c4 against k4}]
For the upper bound, notice that $K_2+\overline{K}_{n-2}$ is $K_4$-saturated with $\binom{n-2}{2}$ copies of $C_4$.
To prove the lower bound, let $G$ be an $n$-vertex graph that is $K_4$-saturated.
If $e(G) > n^{ 5/3 + \delta }$, then the number of $C_4$'s in $G$ is at least 
\[
\frac{ 2 e(G)^4}{ n^4} - \frac{3}{4} e(G) n 
= 
\left(2 n^{3 \delta } - \frac{3}{4} \right) n^{8/3 + \delta } > \frac{n^2}{2}
\]
for large enough $n$ in terms of $\delta$ (see Lemma 2.5 
of \cite{f}).  
Now assume $e(G) \leq n^{5/3 + \delta }$.  
The argument of Proposition \ref{quadratic lower bound} shows that $G$ contains at least 
\[
\binom{n}{2} - e(G) \geq \binom{n}{2} - n^{5/3 + \delta }
\]
copies of $C_4$.
\end{proof}


\section{Counting subgraphs of clique-saturated graphs}\label{section clique saturated}
In this section we focus on graphs which are $K_s$-saturated. In Section \ref{kr against ks}, we prove Theorem \ref{small improvement}. We improve the results for the case $s=4$ and $r=3$ in Theorem \ref{K4triangletheorem} in Section \ref{section K4triangletheorem}. In Section \ref{section c6 against k5} we count cycles to prove Theorems \ref{cr against ks with constants} and \ref{c6 against k5}.

\subsection{Proof of Theorem \ref{small improvement}}
\label{kr against ks}

 
 
  


We begin this section with a lemma that is certainly known, but a proof 
  is included for completeness.  A similar result was proved by Amin, Faudree, and Gould \cite{afg} in the case that $s = 4$.    
  
\begin{lemma}\label{min degree s-2}
Let $ n > s \geq 3$ be integers.  
If $G$ is a $K_s$-saturated graph with $\delta (G) \leq s-2$, then $G$ is isomorphic to $K_{s-2} + \overline{K}_{n - s + 2}$.
\end{lemma}
\begin{proof}
Suppose $G$ is a $K_s$-saturated graph with $n$ vertices.
Note that $\delta (G) \leq s-2$ implies $\delta (G) = s - 2$
since $G$ is $K_s$-saturated.     
Choose a vertex 
$v$ with $d(v) = s-2$.  
If $u \in N_2(v)$, then $N(v) \cap N(u)$ must contain a 
$K_{s-2}$, but since $|N(v)| = s - 2$, $N(v)$ must
then be a clique.  
If $u_1$ and $u_2$ are distinct vertices in $N_2(v)$, then $u_1$ cannot be adjacent to 
$u_2$, otherwise $\{u_1 , u_2 \} \cup N(v)$ is a $K_s$ in $G$.  
This shows that $G$ contains a copy $K_{s-2} + \overline{K}_{n - s + 2}$ 
where $N_2(v) \cup \{v \}$ is the independent 
set of size $n - s + 2$.  
The graph $K_{s-2} + \overline{K}_{n - s + 2}$ is $K_s$-saturated and has $n$ vertices, 
so $G$ must be this graph.   
 \end{proof}
  
\bigskip

The graph $K_{s-2} + \overline{K}_{n -s + 2}$ has the property that 
$n-s + 2$ vertices have exactly one $K_{s-2}$ in their neighborhood.
The next lemma shows that this cannot occur when there are no vertices of degree $s-2$.  

\begin{lemma}\label{no single copy}
Let $n \geq 2s - 2$ and $s  \geq 3$ be integers.  
If $G$ is a $K_s$-saturated graph on $n$ vertices with $\delta (G) \geq s-1$, then 
no vertex has just one copy of $K_{s-2}$ in its neighborhood.
\end{lemma}
\begin{proof}
Suppose $G$ is a $K_s$-saturated graph with $n$ vertices and $\delta (G) \geq s-1$.  
Aiming for a contradiction, assume $v$ is a vertex with exactly one copy of $K_{s-2}$ in 
$N(v)$.  Let $S \subseteq N(v)$ be the vertices that induce the unique $K_{s-2}$ in $N(v)$.

\medskip
\noindent
\textit{Case 1}: $d(v) \leq n -2$
\medskip

For any vertex $u \in N_2 (v)$, $N(u) \cap N(v)$ contains a $K_{s-2}$.  By uniqueness, 
this $K_{s-2}$ must be $S$.  
This implies $S \subseteq N(u) \cap N(v)$ and in 
particular, $u$ is adjacent to all vertices in $S$.  As every vertex in $N_2 (v)$ is joined to $S$, the 
set $N_2(v)$ must be an independent set, otherwise $G$ contains a $K_s$.  
By assumption, $d(v) \geq s - 1$ and so there is a vertex $v' \in N(v)$ with $v' \notin S$.  As there 
is only one $K_{s-2}$ in $N(v)$, vertex $v'$ cannot be adjacent to all vertices in $S$.  Say $v'$ is not 
adjacent to $v_1 \in S$.  The set $N(v_1 ) \cap N(v')$ must contain a $K_{s-2}$.  Let 
$S'$ be the vertices of such an $(s-2)$-clique. 
Note $v_1 \notin S'$ and $v' \notin S'$ since $S' \subseteq N(v_1) \cap N(v')$.  
 If $ |S' \cap N(v)| \geq s - 3$, then there is more than 
one $K_{s-2}$ in $N(v)$.  Indeed, $( S' \cap N(v) ) \cup \{ v' \} $ 
would contain a $K_{s-2}$ in $N(v)$ 
different from $S$.  This also shows $v \notin S'$
otherwise, $S' \subseteq \{v \} \cup N(v)$.  
Since the lemma is trivially true for $s=3$, assume that $s\geq 4$.
As $|S' \cap N(v) | \leq s-  4$ and 
$v \notin S'$, 
$S'$ contains at least two vertices in $N_2(v)$.  This contradicts 
the fact that $N_2(v)$ is an independent set.  

\medskip
\noindent
\textit{Case 2}: $d(v) = n -1$
\medskip
 
Let $W$ be the neighbors of $v$ that are not in $S$.  First suppose there is a pair of nonadjacent vertices, 
say $w_1$ and $w_2$, in $W$.  
Then $N(w_1) \cap N(w_2)$ must contain a $K_{s-2}$, say $S'$ are 
the vertices of such a $(s-2)$-clique.  If $v \notin S'$, then $S' = S$, but then we can remove a vertex from 
$S$ and replace it with $w_1$ to get a $K_{s-2}$ in $N(v)$ different from $S$.    
Therefore, $v$ must be in $S'$ and $|S' \backslash \{v \}  | = s- 3$.  
But then $w_1 \cup S'$ is an $(s-2)$-clique 
in $N(v)$ that is different from $S$.  This shows that $W$ is a clique and so $|W| < s - 1$ as $G$ is $K_s$-free.  This contradicts the 
assumption that $n \geq 2s - 2$.  

 \end{proof}

\bigskip

Lemma \ref{no single copy} shows that the neighborhood of any vertex 
in a $K_s$-saturated graph $G$ with $\delta (G) \geq s - 1$ must have 
at least two copies of $K_{s-2}$ in its neighborhood.  
We now use this lemma to characterize $K_s$-saturated graphs with $\delta (G) = s - 1$.

For integers $n > s \geq 3$, let $(K_{s-1} -e ) + \overline{K}_{n - s + 1}$ be the graph obtained 
by taking a $K_{s-1}$ and removing an edge $e$, and then joining all vertices of 
this graph to an independent set of size $n - s + 1$.
This graph is the same as the complete $(s-1)$-partite graph with part sizes 
$1,1, \dots , 1$ ($s-3$ times), 2, and $n - s+ 1$.  

 Let $W$ be the 6-vertex graph obtained by taking 
a 5-cycle $a_1 a_2 a_3 a_4 a_5 a_1$ and joining a new vertex $b$ to each vertex on the 5-cycle.
We call $b$ the \emph{central vertex}.  
 For $s \geq 3$ and positive integers $m_1$, $m_3$, $m_4$ with 
$m_1 + m_3 + m_4 = n - s + 1$, let $W_s( m_1, 1 , m_3 , m_4 , 1)$ be the graph
obtained from $W$ by replacing $a_i$ with an independent set $I_i$ with $|I_i| = m_i$ ($i = 1,3,4$), 
and replacing the central vertex $b$ with a clique of size $s-3$.  If $x$ and $y$ are vertices that replaced $a_i$ and $a_j$, 
respectively,
then $x$ and $y$ are adjacent if and only if $a_i$ and $a_j$ are adjacent in $W$.  
Vertices in the $(s-3)$-clique that replaced the central vertex $b$ are adjacent to all vertices in the graph and 
so have degree $n-1$.  

Amin, Faudree, and Gould \cite{afg} showed that if $G$ is an $n$-vertex $K_4$-saturated graph 
that is 3-connected, then $G$ is isomorphic to $(K_{3} - e ) + \overline{K}_{n - 3}$, or to 
$W_4( m_1 , 1, m_3 , m_4 , 1)$ for some $m_1 + m_3 + m_4 = n - 3$.    We prove a similar result 
for $K_s$-saturated graphs that have minimum degree $s-1$.

\begin{lemma}\label{lemma for degree}
If $G$ is a $K_s$-saturated $n$-vertex graph with $\delta (G) = s - 1$, 
then $G$ is isomorphic to $(K_{s-1} -e ) + \overline{K}_{n - s + 1}$, or to 
$W_s (m_1 , 1 , m_3 , m_4 , 1 )$ for some $m_1 + m_3 + m_4 = n - s+1$.  
\end{lemma}
\begin{proof}
Suppose $v$ is a vertex in a $K_s$-saturated $n$-vertex graph $G$ where $\delta (G) = s -1$
and $d(v) = s-1$.
By Lemma \ref{no single copy}, there must be at least two $(s-2)$-cliques in $N(v)$.
If there are more than two $(s-2)$-cliques in $N(v)$, then $N(v)$ is complete, which gives a $K_s$ in $G$.
Thus, $N(v)$ contains exactly two $(s-2)$-cliques.  
Let $S_1 = \{v_1, v_2 , \dots , v_{s-3} , v_{s-2} \}$ be the first $K_{s-2}$, and 
$S_2 = \{ v_1 , v_2  , \dots , v_{s - 3} , v_{s-1} \}$ be the second (so the only edge missing 
from $N(v)$ is $v_{s-2} v_{s-1}$).  

Let $T_1$ be all vertices in $N_2(v)$ that are adjacent to
every vertex in $S_1$, but not adjacent to $v_{s-1}$.  
Similarly, let $T_2$ be all vertices in $N_2(v)$ that are adjacent to 
all vertices in $S_2$, but not adjacent to $v_{s-2}$.  
Lastly, let $T_3$ be all vertices in $N_2(v)$ 
that are adjacent to all vertices in $S_1 \cup S_2 = N(v)$.
Since $N(v) \cap N(t)$ must contain a $K_{s-2}$ for any $t \in N_2(v)$, 
the sets $T_1$, $T_2$, and $T_3$ form a partition of $N_2(v)$.  
Also, both $T_1 \cup T_3$ and $T_2 \cup T_3$ are independent sets since $G$ is $K_{s}$-free.

If $T_1 = T_2 = \emptyset$, 
then $T_3$ is an independent set on $n -s+1$ vertices
that are all joined to each vertex in 
$\{v_1, v_2, \dots , v_{s-2}, v_{s-1} \}$.  Vertex $v$ is also joined 
to these vertices, but is not joined to any vertex in $T_3$. 
This shows that $G$ contains a subgraph isomorphic to 
$(K_{s-1} -e ) + \overline{K}_{n - s + 1}$.  This last graph is $K_s$-saturated and has $n$ 
vertices so $G$ must be this graph.

Now suppose $T_1 \neq \emptyset$ and let $t \in T_1$.  Since $t$ is not adjacent to 
$v_{s-1}$, $N(t) \cap N( v_{s-1} )$ must contain a $K_{s-2}$.  
The intersection $N(t) \cap N( v_{s-1})$ contains the $(s-3)$-clique $\{v_1, v_2 , \dots , v_{s-3} \}$ so there 
must be another vertex $x$ for which $x$ is adjacent to both $t$ and $v_{s-1}$.
If $x \in T_1 \cup T_3$, then we contradict the fact that $T_1 \cup T_3$ is an independent set.
Therefore, $x \in T_2$ and so $T_2 \neq \emptyset$.
This argument shows that $T_1 \neq \emptyset$ if and only if $T_2 \neq \emptyset$.  
Next, let $y \in T_1$ and $z \in T_2$ be arbitrary vertices.  We will show that $y$ and $z$ are adjacent. 
If they are not, then $N(y) \cap N(z)$ must contain a $K_{s-2}$.  Now $N(y) \cap N(z) \cap N(v) = \{v_1, v_2 , \dots , v_{s-3} \}$,
and so there must be a vertex in $N_2(v)$ that is adjacent to both $y$ and $z$.   
This is impossible though since $y \in T_1 \cup T_3$, $z \in T_2 \cup T_3$, 
$T_1 \cup T_3$ and $T_2 \cup T_3$ are independent sets, and 
$N_2(v) = T_1 \cup T_2 \cup T_3$.  Thus, every vertex in $T_1$ is joined to every vertex in $T_2$. 
At this point, we have a $K_s$-saturated subgraph that is isomorphic to 
\[
W_s ( |T_3| + 1 , 1, |T_1 | , | T_2 | , 1).  
\]
Indeed, $\{v_1 , v_2 , \dots , v_{s-3} \}$ is a $(s-3)$-clique and every vertex in this set has degree $n-1$.
If this clique replaces the central vertex $b$ in the graph $W$ defined before Lemma \ref{lemma for degree}, 
and we replace $a_1$ with $T_3 \cup \{v \}$, 
$a_2$ with $v_{s-2}$, $a_3$ with $T_1$, $a_4$ with $T_2$, and $a_5$ with $v_{s-1}$, we obtain 
a $W_s ( m_1 , 1 , m_3 , m_4 , 1)$.  
This last graph is $K_s$-saturated and has $n$ vertices, so
$G$ must be this graph.   
\end{proof}

\bigskip
Let us summarize what we have shown so far.  Let $G$ be an $n$-vertex $K_s$-saturated graph.
\begin{enumerate}
\item If $\delta (G) \leq s - 2$, then $G$ is isomorphic to 
$K_{s-2} + \overline{K}_{n - s + 2}$.  
\item If $\delta (G) = s - 1$, then 
$G$ is isomorphic to $(K_{s-1} -e ) + \overline{K}_{n - s + 1}$, or some 
\[
W_s ( m_1 , 1, m_3 , m_4 , 1)
\]
with $m_1 + m_3 + m_4 = n - s + 1$.    
\end{enumerate}

We now use Lemmas \ref{min degree s-2},
\ref{no single copy}, and \ref{lemma for degree}
to 
prove Theorem \ref{small improvement}.
  
\bigskip  

\begin{proof}[Proof of Theorem \ref{small improvement}]
Let $G$ be a $K_s$-saturated graph with $n$ vertices. 
We first show that there are at least 
\[
\frac{1}{r} \left( \binom{s-2}{r-1} + \binom{s-3}{r-2} \right)n 
\]
copies of $K_r$ in $G$.  

If $\delta (G) = s- 2$, then 
$G$ is isomorphic to $K_{s-2} + \overline{K}_{n - s + 2}$ by Lemma \ref{min degree s-2}.  
This graph has 
$\binom{s-2}{r} + (n - s + 2) \binom{s - 2}{r-1}$ 
copies of $K_r$.  For large enough $n$, 
this is at least 
$\frac{1}{r} \left( \binom{s-2}{r-1} + \binom{s - 3}{r-1}  \right) n$.  

If $\delta (G) =  s  -1$, then by Lemma \ref{lemma for degree}, $G$ is isomorphic to 
$(K_{s-1} -e ) + \overline{K}_{n - s + 1}$ or $W_s (m_1 , 1, m_3  , m_4 , 1)$ for some 
$m_1 + m_3 + m_4 = n - s + 1$.  The first graph has 
\[
\binom{s-2}{r} +  \binom{s-3}{r-1} + 
(n - s +1) \left( \binom{s-2}{r-1} + \binom{s-3}{r-2} \right)
\]
copies of $K_r$.  
A member of $W_s (m_1 , 1 , m_3 , m_4 , 1)$ that minimizes the number of $K_r$'s is 
obtained when two of the $m_i$'s are 1, and the other is $n - s - 1$.  The number of $K_r$'s in this graph is 
\[
(n - s -1) \left( \binom{s-2}{r-1} + \binom{s-3}{r-1} \right) 
+ \binom{s-3}{r} + 4 \binom{s-3}{r-1} + 3 \binom{s-3}{r-2}.
\] 
In both cases, we have at least 
$\frac{1}{r} \left( \binom{s-2}{r-1} +
\binom{s-3}{r-2} \right) n$ copies of $K_r$ for large enough $n$.

Assume $\delta (G) \geq s $.  By Lemma \ref{no single copy}, every vertex has at least two distinct copies of 
$K_{s-2}$ in its neighborhood.  Thus, for all $v \in V(G)$, 
the number of $K_{r-1}$'s in $N(v)$ is at least 
\[
\binom{s-2}{r-1} + \binom{s-3}{r-2} 
\]
as 
the two $(s-2)$-cliques in $N(v)$ cannot form a $K_{s-1}$ (this would create a $K_s$, using $v$, in $G$).
The number of $K_r$'s in $G$ is at least 
 \[
 \frac{1}{r} \sum_{v \in V(G) } 
 ( \mbox{number of $K_{r-1}$'s in $N (v)$}) 
 \geq \frac{1}{r} \left( \binom{s-2}{r-1} + \binom{s-3}{r-2}  \right)n.
 \]
 Next we show that there are also at least 
 \[
 \frac{1}{r-1} \binom{s-2}{r - 1}n - \binom{s-2}{r-1} - o_n(1)
\]
copies of $K_r$ in $G$.  
By a result of Erd\H{o}s \cite{erdos paper} for $r = 3$ and 
Mubayi \cite{mubayi} for $r \geq 4$, there 
is a positive constant $\alpha_{r,s}$, depending 
only on $r$ and $s$, such that if $G$ has at least $\textup{ex}(n , K_r ) + \alpha_{r,s}$ edges, then $G$ has at least $\binom{s-2}{r-1}n$ copies of $K_r$,
in which case we are done.
Now assume that 
\[
e(G) \leq \textup{ex}(n , K_r) + \alpha_{r,s} 
\leq \left( 1 - \frac{1}{r-1} \right) \frac{n^2}{2} + \alpha_{r,s}.
\]
Consider a pair of nonadjacent vertices $x$ and $y$.  Their common 
neighborhood contains at least one copy of $K_{s-2}$ since $G$ is 
$K_s$-saturated.  This gives $2 \binom{s-2}{r-1}$ copies of $K_r$ 
that contain the vertex $x$ or contain the vertex $y$.
Now $x$ has at least $s-2$ neighbors, and so each copy of 
$K_r$ containing $x$ obtained in this way (by choosing a 
nonedge containing $x$ and looking at the common neighborhood)
is counted at most $n - s + 2$ times.  Thus, 
the number of copies of $K_r$ in $G$ is at least
\[
\frac{ 2 \binom{s- 2}{r- 1} e( \overline{G} ) }{n - s  +2}
\geq
\frac{2 \binom{s-2}{r-1} }{n} \left( \binom{n}{2} - e(G) \right)
\geq
\frac{2 \binom{s-2}{r-1} }{n} \left( \frac{n^2}{2 (r - 1) } - \frac{n}{2} - \alpha_{r,s} \right).
\]
For large enough $n$, this is at least 
\[
\frac{1}{r-1} \binom{ s-2 }{r-1} n - 2 \binom{s-2}{r-1}.
\]
\end{proof}

\subsection{Proof of Theorem \ref{K4triangletheorem}}\label{section K4triangletheorem}

Let $G$ be a $K_4$-saturated graph on $n$ vertices. 
We must show that $G$ has at least $n -2$ triangles, and 
if $G$ has $n-2$ triangles, then $G$ is isomorphic to $K_2 + \overline{K}_{n-2}$.
A \emph{triangle block} is a maximal subgraph of $G$ constructed by starting with a triangle and repeatedly adding triangles to it such that each new triangle shares at least one edge with a previous triangle. One can easily see that if a triangle block contains $x$ vertices, then it contains at least $x-2$ triangles. In fact, $K_2 + \overline{K}_{n-2}$ is a triangle block on $n$ vertices. Also notice that if two triangle blocks have at least two vertices in common, and their union contains $x$ vertices, then it contains at least $x-2$ triangles. 

A \emph{triangle cluster} is a maximal union of triangle blocks $B_1, B_2, \ldots, B_k$ such that each block $B_i$ (for $2 \le i \le k$) shares at least two vertices with the union of blocks $B_1, B_2, \ldots, B_{i-1}$. A triangle cluster also has the property that if it has $x$ vertices, then it has at least $x-2$ triangles. More importantly, note that any two triangle clusters share at most one vertex in common.  Indeed, otherwise their union is contained in a triangle cluster contradicting the maximality. 

\begin{claim}
\label{notight3path}
If a triangle cluster $C$ has three triangles of the form $abc, bcd, cde$, then $G$ has more than $n-2$ triangles.
\end{claim}
\begin{proof}
Consider any vertex $v$ not in $C$. Then $v$ is not adjacent to at least three of the vertices $x,y,z \in \{a,b,c,d,e\}$. Now notice that any two non-adjacent vertices $p$ and $q$ belong to the same triangle cluster.  Indeed, adding $pq$ to $G$ must create a $K_4$, so there exist vertices $r,s$ such that $prs$ and $qrs$ are triangles in $G$, so $p$ and $q$ belong to the same triangle block, and so they belong to the same triangle cluster as well. Suppose $v$ and $x$ belong to a triangle cluster $C_1$, 
$v$ and $y$ belong to $C_2$, and $v$ and $z$ belong to $C_3$. 
Now $C_1$, $C_2$, $C_3$ are distinct triangle clusters because 
if, say $C_1 = C_2$, then $C_1$ and $C$ would share two vertices ($x$ and $y$), a contradiction. This implies that every vertex $v$ not in $C$ belongs to at least three different triangles. Suppose $C$ has $m$ vertices and let $t(u)$ denote the number of triangles containing a vertex $u$.  Then since $C$ contains at least $m-2$ triangles, we have $\sum_{u} t(u) \ge 3(n-m)+3(m-2)+1 > 3(n-2)$. On the other hand, the sum $\sum_{u} t(u)$ counts each triangle 3 times exactly, proving the claim.
\end{proof}

\bigskip

By Claim \ref{notight3path}, we can assume that 
every triangle cluster $C$ consists of triangles of the form $abx_1, abx_2, \ldots, abx_r$ for integer $r \ge 1$. If $r \ge 2$, let us call $ab$ the base of a triangle $abx_i$ (for any $i \in \{1, 2, \ldots, r\}$)  and $x_i$ as its tip. For a vertex $u$, let us define $p(u)$ as the number of triangles whose tip is the vertex $u$. If there is a triangle cluster with $n$ vertices, then $G$ is 
isomorphic to $K_2 + \overline{K}_{n-2}$ and we are done.  Assume this is not the case.  

\begin{claim}
\label{tip_degree}
For any vertex $v$, there is a triangle cluster that does not contain $v$. Moreover, $p(v) \ge 2$. 
\end{claim}
\begin{proof}
Consider any triangle cluster $C$ and a vertex $u$ not in $C$.
If we take any triangle $abc$ in $C$, then $u$ is not adjacent to at least two of the vertices $a, b, c$, otherwise $u$ would have to be in $C$. Suppose without loss of generality that $u$ is not adjacent to $a$ or $b$. 
Therefore, the vertices $u$ and $a$ belong to a triangle cluster $C'$, and the vertices $u$ and $b$ belong to a triangle cluster $C''$. Then, by the linearity of triangle clusters, 
$C$, $C'$, $C''$ are distinct and there is no vertex contained in all three of them, proving the first part of the claim. Thus, for any vertex $v$, there is a triangle cluster $D$ not containing it; moreover $v$ is not adjacent to some two vertices $a, b$ in $D$. The second part of the claim simply follows by using the fact that adding the pairs $va$ or $vb$ must create a $K_4$ in $G$.
\end{proof}

\bigskip

By Claim \ref{tip_degree}, $\sum_u p(u) \ge 2n$. Moreover, the sum $\sum_u p(u)$ counts each triangle at most once (notice that the triangles that do not share an edge with another triangle are not counted by this sum). So the number of triangles in $G$ is at least $2n$. Since if $G$ has $2n\geq n-2$ triangles when $G\neq K_2+\overline{K}_{n-2}$, we conclude that $\textup{sat}(n,K_3,K_4)=n-2$ and the extremal example is uniquely achieved by $K_2+\overline{K}_{n-2}$. 



\subsection{Proof of Theorems \ref{cr against ks with constants} and \ref{c6 against k5}}\label{section c6 against k5}

\begin{proof}[Proof of Theorem \ref{cr against ks with constants}]
For an upper bound on $\sat(n,C_r,K_{s})$ consider the graph $G=K_{s-2}+\overline{K}_{n-s+2}$. Let $A$ be an independent set of  $k=\lfloor\frac{r}{2}\rfloor$ vertices in $\overline{K}_{n-s+2}$. There are $\binom{n-s+2}{k}$ ways to pick an independent set of size $k$. If $r$ is even, then there are $\frac{(s-2)_{k}(k-1)!}{2}$ $C_r$ subgraphs containing $A$ and each $C_r$ is counted once. 
If $r$ is odd, then there are $\frac{(s-2)_{k+1}k!}{2}$ $C_r$ subgraphs containing $A$ and each $C_r$ is counted once.  
Furthermore, there does not exist a $C_r$ subgraph with more than $\lfloor\frac{r}{2}\rfloor$ vertices in $\overline{K}_{n-s+2}$ since $\alpha(C_r)=\lfloor \frac{r}{2}\rfloor$. Thus, $$\sat(n,C_r,K_{s})\leq \begin{cases} 
\left( \frac{(s-2)_k}{2k}\right)\left(n^{k}+o(n^{k})\right) & \text{if } 2|r\\
\left( \frac{(s-2)_{k+1}}{2}\right)\left(n^{k}+o(n^{k})\right) & \text{if } 2\not |r
\end{cases}.$$

Let $G$ be a graph that witnesses $\sat(n,C_r,K_s)$ for $s\geq 5$, $2s-4\geq r\geq s+1$. Notice that if $xy\notin E(G)$, then there exists a $K_{s-2}$ subgraph in the common neighborhood of $x$ and $y$. Furthermore, if $xy\notin E(G)$, then there exists $s-2$ internally disjoint $x,y$-paths of length $2$.

\medskip
\noindent
\textit{Case 1}: $r$ is even.
\medskip

Let $A\subseteq V(G)$ be an independent set of size $k$. Fix the index for $a_1\in A$. There are $(k-1)!$ ways to index the remaining elements. Notice that for each $1 \le i \le k-1$, the vertices $a_i,a_{i+1}$ have at least $s-2$ common neighbors (since there is a copy of $K_{s-2}$ in their common neighborhood), and similarly $a_1, a_k$ have at least $s-2$ common neighbors.

If $r$ is even, then we will select distinct $b_i\in N(a_i)\cap N(a_{i+1})$ for $1\leq i \leq k-1$ and $b_k\in  N(a_1)\cap N(a_{k})$. So we pick $k$ different elements to form a set $B=\{b_i: 1\leq i \leq k\}$. Since $r\leq 2s-4$, we have that $k\leq s-2$; so we can always pick $B$ in at least $(s-2)_k$ ways. Since $$a_1b_2a_2\dots a_kb_k$$ is a cycle of length $r$, the total number of cycles of length $r$ we see is at least $\frac{(s-2)_k(k-1)!}{2}$ times the number of independent sets of size $k$ in $G$.


By Theorem 1$^{**}$ in \cite{ES_1983}, there exists $c,c'>0$ such that for any graph $G$ with $$|E(G)|\geq cn^{2-2/r},$$ there exists $$c'n^{r}\left(\frac{|E(G)|}{n^2}\right)^{(r/2)^2}$$ copies of $K_{r/2,r/2}$. Each copy of $K_{r/2,r/2}$ contains many copies of $C_r$. Therefore, if $|E(G)|=\epsilon n^2$ and $n$ is sufficiently large, there are $\Theta(n^r)$ copies of $C_r$. Thus, we can assume that $|E(G)|=o(n^2)$ and that $G$ has $n^2/2-o(n^2)$ non-edges. Using the Moon-Moser Theorem, we know that $G$ has at least $\binom{n}{k}-o(n^k)$ independent sets of size $k$.

Each $C_r$ in $G$ is counted at most $2$ times ($C_r$ induces at most two independent sets of size $k$). Putting our estimates together gives, $$\sat(n,C_r,K_s)\geq\left(\frac{(s-2)_k(k-1)!}{4}\right)\left(\binom{n}{k}-o(n^k)\right)=\left(\frac{(s-2)_k}{4\cdot k}\right)\left(n^k-o(n^k)\right).$$

\medskip
\noindent
\textit{Case 2}: $r$ is odd.
\medskip

Let $xy\notin E(G)$, $Z=\{z_1,\dots, z_{s-2}\}\subseteq N(x)\cap N(y)$ induce a clique, and $A=\{a_2,\dots, a_{k-1}\}\subseteq V(G)\setminus (Z\cup \{x,y\})$. These elements could be indexed in $(k-2)!$ ways.

Construct an $x,y$-path $P$ as follows. Let $a_1=x$ and $a_{k}=y$. For each $1\leq i\leq k$ with $a_ia_{i+1}\notin E(G)$ we can choose $b_i\in N(a_i)\cap N(a_{i+1})$ that has not been used to create an $a_1,a_{k}$-path that traverses the vertices $a_i$ in the order of their index. Let $B$ be the set of the $b_i$'s chosen, $B=\{b_i\in N(a_i)\cap N(a_{i+1}):a_ia_{i+1}\notin E(G) \}$. Notice that $|B|\leq k-1$ since we choose at most one $b_i$ per pair $a_ia_{i+1}$. It follows that  $|V(P)|= k+|B|<r$ and  $|Z\setminus V(P)|= s-2-|B|>0.$ The set $B$ can be chosen in at least $(s-2)_{|B|}$ ways.

Since $|Z\setminus V(P)|=s-2-|B|$ and $|V(P)|=k+|B|$, we can extend $P$ to a $C_r$ subgraph if $r-k-|B|\leq s-2-|B|$. This condition is met by our assumption on $r$. Extending $P$ with $Z$ can be done in $(s-2-|B|)_{k+1-|B|}$ ways. Therefore, for any fixed choice of $xy\notin E(G)$, $Z$, and $A$, we see at least $(k-2)!(s-2)_{k+1}$ copies of $C_s$ in $G$. 

Notice that there are at least $\left(\binom{n}{2}-\textup{ex}(n,K_{s})\right)\geq\left(\frac{n^2}{2(s-1)}-\frac{n}{2}\right)$ ways to choose $xy\notin E(G)$. Then there are $\binom{n-s}{ k-1}$ ways to pick $A$. If we count copies of $C_r$ in this way, each copy of $C_r$ in $G$ is counted at most $\left(\frac{r(r-3)}{2}\right)\cdot\binom{r-2}{k-2}$ times (choose a non-edge in the cycle, and then choose the set $A$). Thus, 
\begin{align*}
\sat(n,C_r,K_{s})&\geq \frac{(k-2)!(s-2)_{k+1}}{\left(\frac{r(r-3)}{2}\right)\cdot\binom{r-2}{k-2}}\left(\frac{n^2}{2(s-1)}-\frac{n}{2}\right)\binom{n-s}{ k-2}\\
&\geq \left(\frac{(s-2)_{k+1}(k-2)!}{r(r-3)(r)_k(s-1)}\right)\left(n^k-o(n^{k})\right).
\end{align*}

\end{proof}


Next we prove Theorem \ref{c6 against k5}.  First we need a lemma 
which relates the number of copies of a $C_6$ in a $K_5$-saturated 
graph to the number of independent sets of size 3.  

\begin{lemma}\label{lemma for c6 in k5 sat}
If $G$ is a $K_5$-saturated graph, then the number of copies of $C_6$ in $G$ is at least 
\[
6 i_3 (G),
\]
where $i_3 (G)$ is the number of independent sets of size 3 in $G$.
\end{lemma}
\begin{proof}
Let $G$ be a $K_5$-saturated graph.  Let $\mathcal{I}_3 (G)$ be the set of all independent sets in $G$ 
having three vertices.  Partition $\mathcal{I}_3 (G)$ into two sets $\mathcal{I}_{3}' (G)$ and $\mathcal{I}_{3}'' (G)$
where $\{x ,  y , z \} \in \mathcal{I}_{3}' (G)$ if and only if there is a triangle in the common neighborhood of $\{x,y,z \}$.
We will count the number of copies of $C_6$ of the form $x\alpha y \beta z \gamma x$ where at least one of $\{x,y,z \}$ or 
$\{\alpha , \beta , \gamma \}$ is an independent 
set in $G$.   Then the only $C_6$'s that will be counted more than once are those for 
which both $\{x,y,z \}$ and $\{ \alpha , \beta, \gamma \}$ are independent sets. (Note that this situation 
does not occur until Case 3 given below.)  These $C_6$'s will be counted twice.  

For vertices $x$ and $y$, let $N(x,y)$ be the vertices adjacent to both $x$ and $y$.  When $x$ and 
$y$ are not adjacent, $N(x,y)$ contains at least one triangle, since $G$ is $K_5$-saturated.

Let $\{x,y,z \} \in \mathcal{I}_{3}' (G)$.  Choose a triangle $abca$ with $a,b,c \in N(x,y) \cap N(x,z) \cap N(y,z)$.
Then we have six copies of $C_6$ using $x$, $y$, $z$, $a$, $b$, and $c$:
\begin{equation}\label{list of c6}
xaybzcx, ~xayczbx, ~xbyazcx, ~xbyczax, ~xcyazbx, ~xcybzax.
\end{equation}
If we choose another $\{x_1, y_1 , z_1 \} \in \mathcal{I}_{3}' (G)$ distinct from $\{x,y,z \}$, and a 
triangle $a_1 b_1 c_1a_1$ that lies in the common neighborhood of $\{x_1 , y_1 , z_1 \}$, then 
none of the 6-cycles 
\begin{center}
$x_1 a_1y_1b_1z_1c_1 x_1$, ~~~ $x_1a_1y_1c_1z_1b_1 x_1$, ~~~$x_1b_1y_1a_1z_1c_1x_1$,
 
$x_1b_1y_1c_1z_1a_1x_1$, ~~~ $x_1c_1y_1a_1z_1b_1 x_1$, ~~~ $x_1c_1y_1b_1z_1a_1x_1$
\end{center}
will coincide with a 6-cycle listed in (\ref{list of c6}).  This is because the only three independent 
vertices in $\{x,y,z,a,b,c \}$ are $x$, $y$, and $z$, and the only three independent vertices 
in $\{x_1 , y_1 , z_1 , a_1 , b_1 , c_1 \}$ are $x_1$, $y_1$, and $z_1$.  Thus, 
we have that the number of copies of $C_6$ counted so far is 
\begin{equation}\label{first count c6}
6 | \mathcal{I}_{3}' (G)|.
\end{equation}
Furthermore, the vertex set of each $C_6$ counted by (\ref{first count c6}) induces a subgraph that is isomorphic
to $K_6 - K_3$ (the graph obtained by removing a triangle from $K_6$).  

Now let $\{x,y,z \} \in \mathcal{I}_{3}'' (G)$.  
We are going to count copies of $C_6$ of the form $x \alpha y \beta z \gamma x$.  The first observation to 
make is that we will not count copies of $C_6$ that are counted by (\ref{first count c6}).
The reason for this is that if $x \alpha y \beta z \gamma x$ is any 6-cycle with $\{x,y,z \} \in \mathcal{I}_3'' (G)$, 
then $\{x , y , z , \alpha , \beta , \gamma \}$ does not induce $K_6 -  K_3$ since 
$x$, $y$, and $z$ have no triangle in their common neighborhood.   
Let us proceed now with the counting.  Let $abca$ be a triangle in $N(x,y)$.
Then $z$ is not adjacent to all $3$ of of $a$, $b$, and $c$ and so we consider 
cases depending on the number of adjacencies between $z$ and $\{a,b,c \}$.

\medskip
\noindent
\textit{Case 1}: $z$ is adjacent to $a$ and $b$, but not $c$

\smallskip
Let $t \in N(x,z)\setminus \{a,b\}$. 
Such a $t$ exists since $x$ and $z$ have at least 3 common neighbors.  We are not claiming that 
$a$, $b$, and $t$ form a triangle.  We now consider two subcases.

First suppose that $y$ is also adjacent to $t$.  Then we have the following list of 12 $C_6$'s:
\begin{center}
$xaybztx$, ~~~ $xaytzbx$, ~~~ $xbyazt x$, ~~~ $xbytzax$, ~~~ $xcyazbx$, ~~~ $xcyaztx$

$xcybzax$, ~~~ $xcytzax$, ~~~ $xcytzax$, ~~~ $xcytzbx$, ~~~ $xtybzax$, ~~~~ $xtyazbx$. 
\end{center}
Each one of these $C_6$'s contains at least two of $a$, $b$, and $c$.  Thus, each is of the 
form $x \alpha y \beta z \gamma x$ where $\{ \alpha , \beta , \gamma \}$ is not an independent set of size 3.  

Now suppose that other than $a$ and $b$, there is no vertex adjacent to each of $x$, $y$, and $z$.  Let 
$s \in N(y,z)$.  Such a vertex exists since $N(y,z)$ must contain a triangle.  
In this case, we have the following list of 13 $C_6$'s:
\begin{center}
$xaybztx$, ~~~ $xayszbx$, ~~~ $xbyaztx$, ~~~ $xbyszax$, ~~~ $xcyazbx$, 

$xcyaztx$, ~~~ $xcybzax$, ~~~ $xcybztx$, ~~~ $xcyszax$, ~~~ $xcyszbx$, 

$xaysztx$, ~~~ $xbysztx$, ~~~ $xcysztx$.
\end{center}
The first ten in the list are of the form $x \alpha y \beta z \gamma x$ where $\{ \alpha , \beta , \gamma \}$ 
is not an independent set of size 3.  

The conclusion is that in both of these subcases, we have at least 10 copies of 
$C_6$ of the form $x \alpha y \beta z \gamma x$ where $\{x,y,z \}$ is an 
independent set, $\{ \alpha , \beta , \gamma \}$ is not an independent set, 
and $\{x,y,z, \alpha , \beta , \gamma \}$ does not induce a graph isomorphic 
to $K_6 - K_3$.

\medskip
\noindent
\textit{Case 2}: $z$ is adjacent to $a$, but not adjacent to $b$ or $c$

\smallskip

Since $N(x,z)$ must contain a triangle, there must be a pair of adjacent vertices $s$ and $t$ with 
$\{ s ,t \} \cap \{a,b,c \} = \emptyset$ and $s,t \in N(x,z)$.  Our goal is to find at least 6 copies of 
$C_6$ of the form $x \alpha y \beta z \gamma x$ where $\{ \alpha , \beta , \gamma \}$ is not an 
independent set.  Four such $C_6$'s are 
\begin{center}
$xbyaztx$, ~~~ $xbyazsx$, ~~~ $xcyaztx$, ~~~~ $xcyastx$.
\end{center}
If $y$ is adjacent to $t$, then two more are $xaytzsx$ and $xaysztx$ and we are done.  
Assume that $y$ is not adjacent to $t$ or $s$.  Since $N(y,z)$ contains a triangle, there is a new 
vertex $u$ with $u$ adjacent to both $y$ and $z$.  Then $xbyuzax$ and $xcyuzax$ are two more 
$C_6$'s with the property that we need.  

The conclusion is that in Case 2, we have at least 6 copies of $C_6$ of the form $x \alpha y \beta z \gamma x$ 
where $\{ x ,y , z \}$ is an independent set, $\{ \alpha , \beta, \gamma \}$ is not an independent set, 
and $\{x,y,z , \alpha , \beta , \gamma \}$ does not induce a graph isomorphic to $K_6 - K_3$.

\medskip
\noindent
\textit{Case 3}: $z$ is not adjacent to $a$, $b$, or $c$

\smallskip

Let $uvwu$ be a triangle in $N(x,z)$ where the vertices $x,y,z,a,b,c,u,v,w$ are all distinct.  

First suppose that $y$ is adjacent to $w$.  Then we have the following 6 copies of $C_6$:
\begin{center}
$xaywzvx$, ~~~~ $xaywzux$, ~~~~ $xbywzvx$, ~~~ $xbywzux$, ~~~ $xcywzvx$, ~~~ $xcywzux$.
\end{center}
Each of these $C_6$'s is of the form $x \alpha y \beta z \gamma x$ where 
$\{x,y,z \}$ is an independent set, $\{ \alpha , \beta , \gamma \}$ is not (they all contain at least two of 
$\{u,v,w \}$), and $\{x ,y , z , \alpha , \beta , \gamma \}$ does not induce a $K_6 - K_3$.  

Now suppose that $y$ is not adjacent to any of $u$, $v$, or $w$.  Let $rstr$ be a triangle in $N(y,z)$ where 
all of the vertices $x,y,z,a,b,c,u,v,w,r,s,t$ are distinct.  
Using these vertices we find 27 copies of $C_6$ that are all of the form $x \alpha y \beta z \gamma x$ because $\alpha$ can 
be any one of $\{a,b,c \}$, $\beta$ can be any one of $\{r,s,t \}$, and $\gamma$ can be any one of $\{ u , v , w \}$.  
While we know that $\{x,y,z \}$ is independent and $\{x,y,z, \alpha , \beta , \gamma \}$ does not induce 
a $K_6 - K_3$, we do not know if $\{ \alpha , \beta , \gamma \}$ is independent.
In the case that $\{ \alpha , \beta , \gamma \}$ is independent, we obtain the 6-cycle 
$x \alpha y \beta z \gamma x$ in two ways; once when we choose $\{x,y,z \} \in \mathcal{I}_{3}''(G)$ and 
again when we choose $\{ \alpha , \beta , \gamma \} \in \mathcal{I}_{3}''(G)$.  
Dividing by two takes care of this over counting and so we obtain at least 27 copies of $C_6$ provided we divide this count 
by 2.

\medskip

Combining Cases 1 through 3, we get at least 
\[
6 | \mathcal{I}_{3}''(G) | 
\]
copies of $C_6$.  
Therefore, the number of $C_6$'s in $G$ is at least 
\[
6 | \mathcal{I}_{3}'(G) | + 6 | \mathcal{I}_{3}'' (G) | = 6 | \mathcal{I}_3 (G) |.
\]
\end{proof}

\begin{proof}[Proof of Theorem \ref{c6 against k5}]
Let $G$ be a graph that witnesses $\sat(n,C_6, K_5)$ with $n$ vertices. First we claim that we may assume $G$ has $O(n^{3/2})$ edges. Suppose that $G$ contains 
at least 
$Cn^{3/2}$ edges for some large enough constant $C > 0$. A supersaturation result of Simonovits (see \cite{ES_1984}) implies that there is a constant $c > 0$ such that $G$ contains at least $c \cdot C^{6}n^3$ many copies of $C_6$. Thus we can choose $C$ large enough to obtain the desired lower bound on the number of $C_6$'s. From now on, we assume $G$ has $O(n^{3/2})$ edges, so the number of pairs $xy\not\in E(G)$ is $\binom{n}{2} - C'n^{3/2}$ for some constant $C'> 0$. 

Now using the Goodman bound on the number of triangles, we know that $G$ has $\binom{n}{3}$ independent sets of size $3$, asymptotically. If $m$ is the number of non-edges in $G$, then 
\begin{align*}\displaystyle
\#(\{x,y,z\}:xy,zy,zx\notin E(G))&\geq \frac{m(4m-n^2)}{3n}\\
&=\frac{(\binom{n}{2} - C'n^{3/2})(4\binom{n}{2} - 4C'n^{3/2}-n^2)}{3n}\\
&=\frac{n^4-o(n^4)}{6n}\\
&=\frac{n^{3}}{6}-o(n^3).
\end{align*}
By Lemma \ref{lemma for c6 in k5 sat}, the number of copies of 
$C_6$ in $G$ is at least 
\[
6 \left( \frac{n^3}{6} - o(n^3) \right) = (1  - o(1)) n^3.
\]
\end{proof}

\section{Cycles in $C_k$-saturated graphs}\label{section cycles against cycles}

The focus of this section is cycles in $C_k$-saturated graphs.  We begin with a few easy propositions.

\begin{proposition}\label{H against odd cycle}
Let 
$n \geq 2k + 2 \geq 4$ be integers.  For any nonbipartite graph $H$, 
\[
\textup{sat}(n , H , C_{2k+1} ) = 0.
\]
\end{proposition}
\begin{proof}
Let $n \geq 2k +2$ and consider a complete bipartite graph where both parts have at least $k+1$ vertices.
This graph is $H$-free and $C_{2k+1}$-saturated.  
\end{proof} 

\begin{proposition}\label{long against short}
(i) For any $n \geq t \geq k \geq 3$, 
\[
\textup{sat}(n , C_t , C_k ) = 0.
\]

\noindent
(ii) Let $r \geq 3$ be an integer.  For any integer $m \geq 1$, 
\[
\textup{sat}( m ( r- 1) + 1 , C_t , C_k ) = 0
\]
whenever $t \geq r +1$ and $r+1 \leq k \leq 2r - 1$.
\end{proposition}
\begin{proof}
First we prove (i).  Given a positive integer $n \geq t$, 
write $n = 1 + q ( k - 2) + r$ where $q$ and $r$ are nonnegative integers 
with $r \in \{0,1, \dots , k - 3 \}$.  
Take $q$ copies of $K_{k-2}$ and one copy of $K_r$, and 
join every vertex in these complete graphs to a new vertex $v$.
This graph is $C_k$-saturated, and has no cycle of length 
greater than $k-1$.  

Next we prove (ii).  Let $t \geq r + 1 $ and $r + 1 \leq k \leq 2r - 1$ 
where $r \geq 3$.  
Consider $m$ copies of $K_r$ identified on a single vertex $v$ (the case when $r = 3$ is the Friendship Graph on $2m +1$ vertices 
and $3m$ edges).  The longest cycle in this graph has length $r$.  When an edge is added in this graph, we obtain a cycle of length $k$
for each $k \in \{r+1 , r+2 , \dots , 2r - 1 \}$.  
\end{proof}

\bigskip

In light of Propositions \ref{H against odd cycle} and \ref{long against short}, we focus our attention on 
\[
\textup{sat}(n , C_t , C_k) 
\]
where $t < k$, and at least one of $t$ or $k$ is even.  
Our arguments used to prove upper bounds on $\textup{sat}(n , C_t , C_k)$ depend on $t$, and so we 
divide this section into some further subsections.


\subsection{Triangles in $C_4$-saturated graphs}\label{k3s in c4-sat}

By Proposition \ref{H against odd cycle},
\[
\textup{sat}(n , K_3  , C_{2k+1} ) = 0
\]
for all $k \geq 1$ and $n \geq 2k +2$.

The first nontrivial case that we consider is  
$\textup{sat}(n , K_3 , C_4)$.  
Note that a $C_4$-saturated graph has diameter at most 3, otherwise adding an edge between a 
pair of vertices at distance more than 3 does not create a $C_4$.  
Both the 5-cycle and the Petersen graph 
have girth 5 and so are $K_3$-free and $C_4$-free.  One can also check that 
these two graphs are $C_4$-saturated.  Both are examples of Moore graphs, and the next 
proposition makes this connection between triangle-free $C_4$-saturated graphs 
and Moore graphs precise.  

\begin{proposition}\label{moore graphs}
Let $G$ be a triangle-free $C_4$-saturated graph. Then either $G$ has diameter $3$, or $G$ is a Moore graph.
\end{proposition}

\begin{proof}
Suppose that $G$ has diameter $2$ and let $x$ and $y$ be non-adjacent vertices 
and  $v\in N(x)\cap N(y)$. Since $G$ is $C_4$-saturated, there exists an $x,y$-path $P$ of length $3$. Since $G$ is triangle free, the path $P$ does not contain $v$. Therefore, $Pv$ is a $C_5$ subgraph of $G$ and the girth of $G$ is $5$. Since $G$ has diameter $2$ and girth $5$, it is a Moore graph \cite{Singleton_1968}.
\end{proof}

\bigskip

The 5-cycle, the Petersen graph, and the Hoffman-Singleton graph 
are all examples of triangle-free $C_4$-saturated graphs with diameter $2$.
For several small values of $n$, 
there are $n$-vertex triangle-free $C_4$-saturated graphs with diameter $3$.  These were found by a computer search and show that 
\[
\textup{sat}(n , K_3 , C_4) = 0 ~~ \mbox{for}~~ n \in \{8, 9 , \dots , 24 \}.
\]
Miller and Codish \cite{cm} investigated extremal graphs of girth at least 5 and at most 32 vertices.  
They determined all graphs with $n$ vertices, girth 5, and the maximum 
number of edges for $n$ in the range $\{20,21, \dots , 32 \}$.  We checked several of their extremal graphs.
Some were $C_4$-saturated while others were not.  For example, the unique extremal graph of girth 5 having 
20 vertices is $C_4$-saturated. 
On the other hand, of the three extremal graphs of girth 5 with 21 vertices, none are $C_4$-saturated.  
All three 
extremal graphs on 22 vertices are $C_4$-saturated, and the largest girth 5 extremal graph found in \cite{cm},  
which has 32 vertices, is also $C_4$-saturated.  These claims were verified using Mathematica \cite{mathematica}.     

We do not know 
if 
$\textup{sat}(n , K_3 , C_4)=0$ for infinitely many $n$, and we were 
unable to show that $\textup{sat}(n , K_3 , C_4) > 0$ for infinitely many $n$.
By taking a vertex of degree $n-1$ and putting a matching of size $\lfloor \frac{n-1}{2} \rfloor$ in its neighborhood, we 
obtain the upper bound
\[
\textup{sat}(n , K_3 , C_4 ) \leq \left\lfloor \frac{n-1}{2} \right\rfloor.
\]

Determining the behavior of this function is an intriguing open problem.


\subsection{Triangles in $C_k$-saturated graphs ($k  \geq 5$) and the proof of Theorem \ref{main theorem k3 against c2k}}

The method we use here is to find a small $K_3$-free $C_{2k}$-saturated graph 
with a special set of vertices.  The presence of this special set of vertices, which will be made precise in a moment, will allow us to 
clone vertices, yet maintain both the $K_3$-free property and the $C_{2k}$-saturated property.   

\begin{lemma}\label{duplicating 2}
Let $k \geq 5$, $d \geq 2$, and $G$ be a $C_k$-saturated graph.  
Suppose that $G$ contains $d$ vertices $v_1 , \dots , v_d$ such that 
$v_i$ and $v_j$ have the same neighborhood for all $1 \leq i , j \leq d$, and this neighborhood has 
size at most $d$.  If $G_u$ is the graph obtained by adding a new vertex $u$ to $G$ and 
making $N(u)$ the same as $N(v_1)$, then $G_u$ is $C_k$-saturated.  
\end{lemma}
\begin{proof}
Suppose $G$ is $C_k$-saturated and has $d$ vertices $v_1 , \dots , v_d$ with the same neighborhood 
$N(v_1)$, and $|N(v_1) | \leq d$.  Let $G_u$ be as in the statement of the lemma.  
If $G_u$ contains a $k$-cycle $C$, then $C$ must contain $u$.  Let $x u $ and $uy$ be the unique edges of $C$ that 
contain $u$.  If a vertex $v_i$ is on $C$, then two vertices in $N(v_i)$ are also on $C$.  Since $u$ is on $C$, at most 
$d-1$ of the vertices $v_1 , \dots , v_d$ can be on $C$ because 
$| N(v_1) | \leq d$.  Without loss of generality, assume that $v_d$ is not on $C$.
Then we can replace $xu$ and $uy$ on $C$ with $x v_d $ and $v_d y$ to get a  $k$-cycle that is in $G$. 
This is a contradiction, so $G_u$ must be $C_k$-free.  To finish the proof, we must show that 
if $w$ is a vertex that is not adjacent to $u$, then there is a path of length $k-1$ from $w$ to $u$.
If $w$ is not adjacent to $u$, then $w$ is not adjacent to $v_1$ and so there is a path $P$ of length 
$k-1$ from $w$ to $v_1$.  We then remove $v_1$ from $P$ and replace it with $u$ to get a path of 
length $k-1$ from $w$ to $u$.   
\end{proof}

\bigskip
Lemma \ref{duplicating 2} is very useful in proving upper bounds on 
$\textup{sat}(n , K_3 , C_{2k})$ because adding a vertex, in the way that is described in Lemma \ref{duplicating 2},
will not create a triangle.  Once we find a $K_3$-free $C_{2k}$-free graph on $m$ vertices 
with a subset of vertices having the same neighborhood as in Lemma \ref{duplicating 2}, 
we get $\textup{sat}(n , K_3 , C_{2k}) = 0$ for all $n \geq m$.  
The next step then is to construct a small $K_3$-free $C_{2k}$-saturated graph.  

Let $k \geq 2$ be an integer.  Let $G(4k)$ be the graph with $4k+2$ vertices whose 
vertex 
set is the disjoint union of five sets 
\[
\{ v ,u_1 , u_2 \} \cup X \cup Y \cup A \cup B
\]
where $X = \{x_1 , \dots , x_k \}$, $Y = \{ y_1 , \dots , y_k \}$, $A = \{a_1 , \dots , a_k \}$, 
$B = \{ b_1 , \dots , b_{k-1} \}$, and 
\begin{itemize}
\item $v$ is a degree 2 vertex with neighbors $u_1$ and $u_2$, 
\item $u_1$ is adjacent to $v$ and all vertices in $X$, and $u_2$ is adjacent to $v$ and all vertices in $A$, 
\item every vertex in $X$ is joined to every vertex in $\{u_1 \} \cup A \cup Y$, and 
\item every vertex in $A$ is joined to every vertex in $\{ u_2 \} \cup X \cup B$.
\end{itemize}
This completes the description of the first graph that is needed.  Now we introduce the second graph which is similar.  

Let $k \geq 1$ be an an integer.  
Let $G(4k+2)$ be the graph with $4k+4$ vertices whose vertex set, like $G(4k)$, is the disjoint union of five sets 
$\{v , u_1 , u_2 \} \cup X \cup Y \cup A \cup B$, except now 
$A = \{ a_1 , \dots , a_k , a_{k+1} \}$ and $B = \{b_1, \dots , b_{k-1} , b_k \}$.  
The adjacencies are defined in the same way as they were defined for $G(4k)$ so the 
edge sets are the same, except in $G(4k+2)$ we also join $b_k$ to every vertex in $A$, and 
join $a_{k+1}$ to every vertex in $\{ u_2 \} \cup X \cup B$.  

\begin{lemma}\label{lemma for G(4k)}
(i) If $k \geq 2$ is an integer, then $G(4k)$ is $K_3$-free and $C_{4k}$-saturated.  

\medskip
\noindent
(ii) If $k \geq 1$ is an integer, then $G(4k+2)$ is $K_3$-free and $C_{4k+2}$-saturated.  
\end{lemma}
\begin{proof}
(i) First we show that for each nonedge $e$ of $G(4k)$, when $e$ is added to $G(4k)$ there is a $4k$-cycle 
that contains $e$.  
There are several cases to consider and it will be extremely 
useful to introduce some notation to make the argument concise.      
Suppose that
\begin{enumerate}
\item $\alpha_1 \alpha_2 $ is a nonedge of $G(4k)$,  
\item $C = \alpha_1 \alpha_2 \alpha_3 \alpha_4 \dots \alpha_{4k} \alpha_1$ is a $4k$-cycle in the graph obtained by 
adding $\alpha_1 \alpha_2$ to $G(4k)$,
\item $\beta_1 $ and $\beta_2$ are the unique pair of vertices in $G(4k)$ not on $C$.
\end{enumerate}
In this case, we will write 
\begin{center}
$\alpha_1 \alpha_2$~:~$\alpha_1 \alpha_2 \alpha_3 \alpha_4 \dots \alpha_{4k} \alpha_1$, 
~~~ $\{ \beta_1 , \beta_2 \}$
\end{center} 
Now we use this notation and list $4k$-cycles we get when adding nonedges to $G(4k)$.

Starting with nonedges on $v$, 

\medskip
$v a_1$~:~
$v a_1 b_1 a_2 b_2 a_3 \dots b_{k-1} a_k x_k y_{k-1} x_{k-1} y_{k-2} \dots y_1 x_1 u_1 v$, 
~~ $\{ y_k, u_2 \}$ 

\medskip
$v b_1$~:~$v b_1 a_1 x_1 y_1 x_2 y_2 \dots y_{k-1} x_k a_k b_{k-1} a_{k-1} b_{k-2} \dots b_2 a_2 u_2 v$, 
~~ $\{ y_k , u_1 \}$

\medskip
$v x_1$~:~$v x_1 y_1 x_2 y_2 x_3 \dots x_{k-1} y_{k-1} x_k a_k b_{k-1} a_{k-1} b_{k-2} \dots a_2 b_1 a_1 u_2 v$, 
~~ $\{ y_k , u_1 \}$

\medskip
$v y_1$~:~$v y_1 x_2 y_2 x_3 \dots x_{k-1} y_{k-1} x_k a_k b_{k-1} a_{k-1} b_{k-2} \dots b_2 a_2 b_1 a_1 x_1 u_1 v$, 
~~ $\{ y_k , u_2 \}$

\medskip

This covers all possible missing edges on $v$ and we no longer check missing edges that contain $v$.
Moving on to those containing $u_1$, the possibilities we must consider are adding the nonedges 
$u_1z$ where $z \in \{ u_2 \}  \cup A \cup B \cup  Y$. 

\medskip
$u_1 u_2$~:~
$u_1 u_2 a_1 b_1 a_2 b_2 a_3 \dots b_{k-1} a_k x_k y_{k-1} x_{k-1} y_{k-2} \dots y_1 x_1 u_1$, 
~~ $\{ y_k , v \}$

\medskip
$u_1 a_1$~:~
$u_1 a_1 b_1 a_2 b_2 \dots a_{k-2} b_{k-2} a_{k-1} x_1 y_1 x_2 y_2 \dots x_{k-1} y_{k-1} x_k a_k u_2 v u_1$, 
~~ $\{ y_k , b_{k-1} \}$

\medskip
$u_1 b_1$~:~
$u_1 b_1 a_1 u_2 a_2 b_2 a_3 b_3 \dots b_{k-1} a_k x_k y_{k-1} x_{k-1} y_{k-2} \dots y_1 x_1 u_1 $, 
~~ $\{ y_k , v \}$

\medskip
$u_1 y_1$~:~
$u_1 y_1 x_2 y_2 x_3 y_3 \dots y_{k-1} x_k a_k b_{k-2} a_{k-1} b_{k-3} \dots a_3 b_1 a_2 x_1 a_1 u_2 v u_1$, 
~~$\{y_k , b_{k-1} \}$

\medskip

This covers all possible missing edges on $u_1$, and we no longer check missing edges that contain $u_1$ or $v$.  
Concerning missing edges on $u_2$, we must check nonedges of the form $u_2 z$ with $z \in X \cup Y \cup B$. 

\medskip

$u_2 x_1$~:~
$u_2 x_1 y_1 x_2 y_2 x_3 y_3 \dots y_{k-2} x_{k-1} a_1 b_1 a_2 b_2 \dots a_{k-1}b_{k-1} a_k x_k u_1 v u_2$, 
~~$\{ y_{k-1} , y_k \}$

\medskip
$u_2 y_1$~:~
$u_2 y_1 x_k y_{k-1} x_{k-1} y_{k-2} \dots y_2 x_2 u_1 x_1 a_k b_{k-1} a_{k-1} b_{k-2} \dots b_2a_2 b_1 a_1 u_2$, 
~~$\{ y_k , v \}$,

\medskip
$u_2 b_1$~:~
$u_2 b_1 a_1 x_1 a_2 b_2 a_3 b_3 \dots a_{k-1} b_{k-1} a_k x_2 y_1 x_3 y_2 x_4 \dots y_{k-2} x_k u_1 v u_2$, 
~~$\{ y_{k-1} , y_k \}$

\medskip

The remaining nonedges all have their endpoints in $A \cup B \cup X \cup Y$.  A careful check 
shows that the list below covers the remaining cases.  

\medskip

$x_1 b_1$~:~
$x_1 b_1 a_2 b_2 a_3 b_3 \dots b_{k-1} a_k x_k y_{k-2} x_{k-1} y_{k-3} x_{k-2} \dots y_2 x_3 y_1 x_2 a_1 u_2 v u_1 x_1$,
~~$\{ y_{k-1} , y_k \}$

\medskip
$x_1 x_2$~:~
$x_1 x_2 y_3 x_3 y_4 x_4 \dots y_k x_k u_1 v u_2 a_1 b_1 a_2 b_2 \dots b_{k-1} a_k x_1$,
~~$\{ y_1 , y_2 \}$

\medskip
$a_1 a_2$~:~
$a_1 a_2 b_1 a_3 b_2 a_4 \dots b_{k-3} a_{k-1} b_{k-2} a_k x_k y_{k-1} x_{k-1} \dots y_2 x_2 y_1 x_1 u_1 v u_2 a_1$,
~~$ \{ y_k , b_{k-1} \}$

\medskip
$a_1 y_1$~:~
$a_1 y_1 x_2 y_2 x_3 y_3 \dots x_{k-1} y_{k-1} x_k a_k u_2 v u_1 x_1 a_{k-1} b_{k-2} a_{k-2} b_{k-3} \dots b_1 a_1$,
~~$ \{ y_k , b_{k-1} \}$

\medskip
$b_1 b_2$~:~
$b_1 b_2 a_2 b_3 a_3 b_4 \dots  b_{k-1} a_{k-1} x_k a_k u_2 v u_1 x_{k-1} y_{k-2} x_{k-2} y_{k-3} 
\dots y_2 x_2 y_1 x_1 a_1 b_1$,
$\{ y_{k-1} , y_k \}$

\medskip
$ y_1 y_2$~:~
$ y_1 y_2 x_3 y_3 x_4  \dots x_{k-1}y_{k-1} x_k a_k b_{k-2} a_{k-1} b_{k-3} \dots b_2 a_3 b_1 a_2 x_1 u_1 v u_2 a_1 x_2 y_1$, 
$\{ y_k , b_{k-1}  \}$

\medskip
$b_1 y_1$~:~
$ b_1 y_1 x_1 y_2 x_2 \dots y_{k-1} x_{k-1} u_1 x_k a_k b_{k-1} a_{k-1} b_{k-2} \dots a_3 b_2 a_2 u_2 a_1 b_1 $, 
~~$\{ v , y_k \}$

\bigskip

We finish the proof of (i) by showing that $G(4k)$ is $C_{4k}$-free.  Suppose, for contradiction, that $C$ is a $4k$-cycle in 
$G(4k)$.  Observe that any cycle of length more than 
$2k$ cannot contain all vertices in $Y$, because the only cycles in $G(4k)$ that contain all vertices in $Y$
are cycles of length $2k$ having $k$ vertices in $X$ and $k$ vertices in $Y$.  
Without loss of generality, assume $y_k$ is not on $C$, and let 
\[
Y' = Y \backslash \{ y_k \}.
\]
If $u_1 $ is not on $C$, then $v$ cannot be on 
$C$, but then $C$ has less than $4k$ vertices.  Thus, $u_1$ is on $C$ and similarly, $u_2$ is also on $C$.
We consider two cases depending on whether or not $v$ is on $C$.

If $v$ is not on $C$, then $C$ must contain all vertices in $\{u_1 , u_2 \} \cup A \cup B \cup X \cup Y'$. 
To contain all vertices in $Y'$, $C$ must have a subpath of length $2k-1$ that starts and ends in $X$, and alternates between vertices in 
$X$ and $Y'$.  By relabeling vertices if necessary, we may assume that 
\[
P = x_1 y_1 x_2 y_2 \dots x_{k-1} y_{k-1} x_k
\]
is this subpath.  Now $u_1$ is on $C$ but $v$ is not on $C$, and the only neighbors of $u_1$ in 
the union $\{u_1 , u_2 \} \cup A \cup B \cup X \cup Y'$ are vertices in $X$.
Therefore, $x_1$ and $x_k$ must be joined to $u_1$ on 
$C$.  This is a contradiction as adding $u_1$ to the endpoints of $P$ closes the cycle $C$ before it touches vertices in $A$.
We conclude that $v$ is on $C$, and so $u_1$ and $u_2$ must also be on $C$.
The rest of the vertices on $C$ are all but one vertex in $A \cup B \cup X \cup Y'$.  
Since $|A| = |X|$ and $|B | = |Y'|$, the sets $X,Y'$ and $A$, $B$ are symmetric.  We may assume, 
without loss of generality that $C$ contains all vertices in $Y'$, otherwise if $C$ misses a vertex in $Y'$ and $B$, 
then $C$ has at most $4k -1$ vertices.  As $C$ contains all vertices in $Y'$, $C$ contains the subpath 
\[
P_1 = a_1 u_2 v u_1 x_1 y_1 x_2 y_2 \dots x_{k-1}y_{k-1} x_k a_2
\]
where we relabel vertices within the sets $A$, $X$, and $Y$ as needed.  The path $P_1$ contains all vertices in $X$ and the edges
touching $u_1$ and $u_2$, so all of the other edges of $C$ that are not in $P_1$ must have one endpoint in $A$ and the other 
in $B$.  If $P_2$ is the subpath of $C$ from $a_2$ to $a_1$ that is different from $P_1$, then $P_2$ must 
contain an even number of edges as it starts and ends in $A$, and alternates between vertices in $A$ and $B$.
This is a contradiction however since $P_1$ has $2k +3$ edges.  

We have shown that no $4k$-cycle exists in $G(4k)$ 
so $G(4k)$ is $C_{4k}$-free.

\bigskip
\noindent
(ii) Now we show that $G(4k+2)$ is $C_{4k+2}$-saturated.  Recall that $G(4k+2)$ is almost the same as 
$G(4k)$, except $A$ and $B$ contain one more vertex each so $A = \{a_1 , \dots , a_k , a_{k + 1} \}$, 
and $B = \{ b_1 , \dots ,b_{k-1},  b_k \}$.  We use the same notation as used in (i),
and we give our list of 
nonedges and $(4k+2)$-cycles containing the nonedges 
here:


\medskip
$v a_1$~:~
$v a_1 u_2 a_2 b_1 a_3 b_2 \dots a_{k-1} b_{k-2} a_k b_{k-1} a_{k+1} x_k y_{k-1} x_{k-1} y_{k-2} \dots x_2 y_1 x_1 u_1 v$,
~~ $\{ b_k, y_k \}$ 

\medskip
$v b_1$~:~$v b_1 a_1 b_2 a_2 \dots b_{k-1} a_{k-1} b_k a_k x_k y_{k-1} x_{k-1} y_{k-2} \dots 
x_2 y_1 x_1 a_{k+1} u_2 v$,
~~ $\{ u_1 , y_k \}$

\medskip
$v x_1$~:~$v x_1 y_1 x_2 y_2 \dots   y_{k-1} x_k a_{k+1} b_k a_k b_{k-1} \dots a_2 b_1 a_1 u_2 v$, 
~~ $\{ u_1 , y_k \}$

\medskip
$v y_1$~:~$v y_1 x_1 a_1 b_1 a_2 b_2 \dots a_{k-1} b_{k-1} a_k b_k a_{k+1} x_2 y_2 x_3 y_3 
\dots x_{k-1} y_{k-1} x_k u_1 v$,
~~ $\{ u_2 , y_k \}$

\medskip
$u_1 u_2$~:~
$u_1 u_2 a_1 b_1 a_2 b_2 \dots a_{k-1} b_{k-1} a_k b_k a_{k+1} x_1 y_1 x_2 y_2 \dots x_{k-1} y_{k-1} x_k u_1$,
~~ $\{ v , y_k \}$

\medskip
$u_1 a_1$~:~
$u_1 a_1 x_1 y_1 x_2 y_2 \dots x_{k-1} y_{k-1} x_k a_{k+1} b_k a_k b_{k-1} a_{k-1} \dots a_3 b_2 a_2 u_2 v u_1$,
~~ $\{ y_k , b_1 \}$

\medskip
$u_1 b_1$~:~
$u_1 b_1 a_1 b_2 a_2 b_3 \dots a_{k-2} b_{k-1} a_{k-1} b_k a_k u_2 a_{k+1} x_1 y_1 x_2 y_2 \dots y_{k-1}x_k u_1$,
~~ $\{ y_k , v \}$

\medskip
$u_1 y_1$~:~
$u_1 y_1 x_1 y_2 x_2 y_3 \dots x_{k-1} y_k x_k a_1 b_1 a_2 b_2 \dots a_{k-1} b_{k-1} a_k u_2 v u_1$,
~~$\{a_{k+1} , b_{k} \}$

\medskip

$u_2 x_1$~:~
$u_2 x_1 a_1 b_1 a_2 b_2 \dots a_k b_k a_{k+1} x_2  y_1 x_3 y_2 \dots x_{k-1} y_{k-2} x_k u_1 v u_2$,
~~$\{ y_{k-1} , y_k \}$

\medskip
$u_2 y_1$~:~
$u_2 y_1 x_1 u_1 x_2 y_2 x_3 y_3 \dots x_{k-1} y_{k-1} x_k a_1 b_1 a_2 b_2 \dots a_k b_k a_{k+1} u_2$,
~~$\{ y_k , v \}$

\medskip
$u_2 b_1$~:~
$u_2 b_1 a_1 b_2 a_2  \dots  b_{k} a_k x_1 y_1 x_2 y_2  \dots x_{k-1} y_{k-1} x_k u_1 v u_2 $, 
~~$\{ a_{k+1} , y_k \}$

\medskip
$x_1 b_1$~:~
$x_1 b_1 a_1 b_2 a_2 b_3 a_3 \dots b_k a_k u_2 v u_1 x_2 y_1 x_3 y_2 \dots x_{k-1} y_{k-2} x_k y_{k-1} x_1$,
~~$\{ y_k , a_{k+1} \}$

\medskip
$x_1 x_2$~:~
$x_1 x_2 y_1 x_3 y_2 x_4 y_3 \dots y_{k-2} x_k a_1 b_1 a_2 b_2 \dots a_k b_k a_{k+1} u_2 v u_1 x_1$,
~~$\{ y_k , y_{k-1} \}$

\medskip
$a_1 a_2$~:~
$a_1 a_2 b_1 a_3 b_2 a_4 \dots a_k b_{k-1} a_{k+1} x_1 y_1 x_2 y_2 \dots x_{k-1} y_{k-1} x_k u_1 v u_2 a_1$,
~~$ \{ b_k , y_k \}$

\medskip
$a_1 y_1$~:~
$a_1 y_1 x_1 y_2 x_2 \dots y_k x_k u_1 v u_2 a_2 b_1 a_3 b_2 \dots a_k b_{k-1} a_1$,
~~$ \{ a_{k+1} , b_k \}$

\medskip
$b_1 b_2$~:~
$b_1 b_2 a_1 b_3 a_2 b_4 \dots a_{k-2} b_k a_{k-1} x_1 y_1 x_2 y_2 \dots x_{k-1} y_{k-1} x_k u_1 v u_2 a_k b_1$,
~~$\{ a_{k+1} , y_k \}$

\medskip
$ y_1 y_2$~:~
$ y_1 y_2 x_1 y_3 x_2 \dots y_k x_{k-1} u_1 v u_2 a_1 b_1 a_2 b_2 \dots a_{k-1} b_{k-1} a_k x_k y_1$,
~~$\{ b_k ,a_{k+1}  \}$

\medskip
$b_1 y_1$~:~
$b_1 y_1 x_1 y_2 x_2 y_3 \dots y_{k-1} x_{k-1} u_1 x_k a_{k+1} u_2 a_k b_k a_{k-1} b_{k-1} \dots a_1 b_1$, 
~~$\{ v , y_k \}$
  
\bigskip

To show $G(4k+2)$ is $C_{4k+2}$-free, we again use proof by contradiction.
Suppose $C$ is a $(4k+2)$-cycle in $G(4k+2)$.  If $C$ contains a subpath of the form $aba'$ where $a , a' \in A$ and $b \in B$, 
then by replacing $aba'$ with $a$, we can obtain a cycle of length $4k$ in $G(4k)$ which we have already shown is impossible.
This means that $C$ cannot contain any vertices in $B$, and we also know that $C$ must miss at least one vertex in $Y$ (the same argument 
used to show this for $G(4k)$ applies here as well since 
$|X| = |Y|$ in $G(4k+2)$).
Thus, $k = |B| \leq 1$. It is then easy to check that $G(6)$ is $C_6$-free.  
\end{proof}

\bigskip

We now have all of the lemmas needed to prove 
Theorem  \ref{main theorem k3 against c2k}.  

\begin{proof}[Proof of Theorem  \ref{main theorem k3 against c2k}]
By the comments preceding 
the statement of Theorem \ref{main theorem k3 against c2k}, we only need to consider cycles of even length. 
That is, we must show that for all $n \geq 2k +2 \geq 6$, 
there is a $K_3$-free $C_{2k}$-saturated graph 
on $n$ vertices. By Lemma \ref{duplicating 2}, it is enough to find a $K_3$-free $C_{2k}$-saturated 
graph with a set of $d \geq 2$ vertices having the same neighborhood whose size is at most $d$.
When $k$ is even, say $k = 2 r$, then the graph $G( 4r )$ has $4r+2$ vertices and is $K_3$-free and 
$C_{2k}$-saturated by Lemma \ref{lemma for G(4k)}.  The vertices in $Y$ form a set of 
$k$ vertices that all have the same neighborhood $X$ which has size $k$.  By Lemma \ref{duplicating 2}, we 
may duplicate vertices in $y$ as many times as needed to obtain a $K_3$-free $C_{2k}$-saturated graph 
on $n \geq 2k+2$ vertices.  When $k$ is odd, say $k = 2r + 1$, the same argument applies except we use the 
graph $G( 4r + 2)$.   
\end{proof}


\subsection{4-cycles in $C_k$-saturated graphs, $k > 6$}

In this subsection, we consider how many $C_4$'s must be in a $C_k$-saturated graph.  We will assume that $k \geq 5$ 
throughout.  
Our approach does not use Lemma \ref{duplicating 2} because 
when a vertex is duplicated, we will create new $C_4$'s.  
Instead, we use 
the idea of $C_k$-builders introduced by Barefoot et.\ al.\ \cite{barefoot}.
 A graph $G$ is a \emph{$C_k$-builder} if $G$ is $C_k$-saturated, and there is a 
 distinguished vertex $v$ in $G$ such that 
 if $v$ in one copy of $G$ is identified with $v$ in the other copy 
 of $G$, then the resulting graph is $C_k$-saturated.  
 Barefoot et.\ al.\ \cite{barefoot} use $C_k$-builders to obtain 
 upper bounds on $\textup{sat}(n , C_k)$ for different values of $k$.  They were also used 
 by Zhang, Luo, and Shigno \cite{zls} in the special case $k = 6$.  
 
 If $G$ is a $C_k$-builder with distinguished vertex $v$ and $G$ is $C_4$-free, 
 then the graph obtained by taking two copies of $G$ and identifying $v$ is $C_k$-saturated and $C_4$-free.
 This observation, like in the case of $\textup{sat}(n , K_3 , C_{2k})$, allows us to reduce the problem 
 to finding a small $C_4$-free $C_k$-builder.   If $G$ is a $C_k$-builder with distinguished vertex $v$,
 then for any ordered pair of vertices $(u_1,u_2)$ where $u_1 \neq v$ and $u_2 \neq v$, 
 there must be positive integers $k_1$, $k_2$ with $k_1 + k_2 = k -1$ 
 and $u_i$ is joined to $v$ by a path of length $k_i$ ($i=1,2$).  We generalize this observation to 
 two different builders in the next lemma.
 
 \begin{lemma}\label{joining builders}
 Let $m_1$ and $m_2$ be positive integers.  
 Let $G_1$ and $G_2$ be $C_k$-builders with distinguished vertices $v_1$ and $v_2$, respectively.
 Suppose for every ordered pair of vertices $(u , w) \in (G_1 \backslash v_1 ) \times (G_2 \backslash v_2)$ 
 there is a path of length $k_1$ from $u$ to $v_1$ in $G_1$, and a path of length $k_2$ from $w$ to $v_2$ 
 in $G_2$ with $k_1 + k_2 = k-1$.  If $G$ is the graph obtained by taking $m_1$ copies of $G_1$ and $m_2$ copies of 
 $G_2$ and identifying each of the $m_1$ copies of $v_1$ and the $m_2$ copies of $v_2$ all into a single vertex $v$, 
 then $G$ is $C_k$-saturated and has 
 \[
 m_1 ( |V (G_1)  | -1) + m_2 ( |V (G_2) |-1) + 1
 \]
 vertices.  Furthermore, if each of $G_1$ and $G_2$ are $H$-free where $H$ is a graph 
 with no cut vertex, then $G$ is also $H$-free.    
 \end{lemma}
 \begin{proof}
 Let $G$ be the graph described in the lemma.  It is clear that $G$ has 
 \[
 m_1 ( |V (G_1)  | -1) + m_2 ( |V (G_2) |-1) + 1
 \]
 vertices.  Consider now a pair of nonadjacent vertices $x$ and $y$ in 
 $G$.  If this pair belongs to the same copy of some $G_i$, $i=1$ or $i=2$, then they are joined by a path of length 
 $k-1$ since $G_i$ is $C_k$-saturated.  
 If $x$ and $y$ are in different copies of $G_i$, then since $G_i$ is a $C_k$-builder, there is a path of length $k_1$ from 
 $x$ to $v$ in the copy of $G_i$ that contains $x$, and a path of length $k_2$ from $y$ to $v$ in the copy of $G_i$ that contains $y$, 
 where $k_1 + k _2 = k-1$.  Finally, assume that 
 $x$ is in a copy of $G_1$ and $y$ is in a copy of $G_2$.  Then, by hypothesis, 
 there is a path of length $k-1$ from 
 $x$ to $y$ that uses the vertex $v$.  
 
 Lastly, since $v$ is a cut-vertex, any 
 copy of $H$ in $G$ must be contained in some copy of $G_1$ or $G_2$.  
 \end{proof}

\bigskip

Let us call a pair of $C_k$-builders satisfying the conditions of Lemma \ref{joining builders} \emph{compatible}.  

\begin{lemma}\label{coprime builders}
Let $H$ be a graph with no cut vertex and 
$k \geq 3$ be an integer. If $G_1$ and $G_2$ are 
$H$-free $C_k$-builders that are compatible and 
$|V(G_1) |-1$ is relatively prime to $|V(G_2) |-1$, 
then 
\[
\textup{sat}(n , H , C_k) = 0
\]
for all $n  \geq n_0$ where $n_0$ depends 
only on $|V(G_1)|$ and $|V(G_2)|$.
\end{lemma}
\begin{proof}
By Lemma \ref{joining builders}, 
the graph obtained by identifying the 
distinguished vertices in $m_1$ copies of $G_1$ and 
$m_2$ copies of $G_2$ into a single vertex is $C_k$-saturated.
This graph is also $H$-free since each of the 
builders $G_1$ and $G_2$ are $H$-free, so no copy of 
$H$ is contained in a single copy of a builder.
If we find a copy of $H$ whose vertices are
in more than one builder, then $H$ contains a 
cut vertex which is not possible. Therefore, 
we have an $H$-free $C_k$-saturated graph on 
\begin{equation}\label{builder eq}
 m_1 ( |V (G_1)  | -1) + m_2 ( |V (G_2) |-1) + 1
\end{equation}
vertices.  Since $|V(G_1)|-1$ and $|V (G_2) |-1$ are 
relatively prime, all sufficiently large positive 
integers can be written in the form (\ref{builder eq}) 
for some nonnegative integers $m_1$ and $m_2$.  
\end{proof}

\bigskip

\begin{proof}[Proof of Theorem \ref{main theorem c4 against ck}]
By Lemma \ref{coprime builders}, it is enough 
to find compatible $C_4$-free $C_k$-builders such that the respective number of vertices minus one are coprime.
The adjacency matrices of 
$C_4$-free compatible $C_k$-builders for $k \in \{7,8,9,10 \}$ 
are given in the appendix.  
The computations 
establishing that the corresponding graphs 
have the needed properties was done using Mathematica 
\cite{mathematica}.  
\end{proof}

\bigskip

\textbf{Remark:} The lower bound on $n$ in Theorem \ref{main theorem c4 against ck} comes from the 
number of vertices in the compatible $C_k$-builders.  The worst case is $k = 10$ where 
our builders have 12 and 13 vertices.  A short computation shows that every integer $n \geq 111$ 
can be written in the from $1 + 11m_1 + 12 m_2$ for some nonnegative integers $m_1$ and $m_2$.  
As mentioned in the introduction, we have verified computationally that 
$\textup{sat}(n , K_3 , C_7 )  = 0$ and $\textup{sat}(n , K_3 , C_8) = 0$ for the cases 
not covered by Theorem \ref{main theorem c4 against ck} \cite{mathematica}.


\subsection{4-cycles in $C_6$-saturated graphs}

In this subsection we discuss 4-cycles in 
$C_6$-saturated graphs.  Like 
$\textup{sat}(n , K_3 , C_4)$, we were unable 
to show that $\textup{sat}(n , C_4 , C_6) > 0$ 
for infinitely many $n$.  Using a computer search, we were able to find $C_4$-free $C_6$-saturated 
graphs for $n \in \{14, 15,18, 20,22, 26 \}$.  
The graphs on 26 vertices that 
are $C_4$-free and $C_6$-saturated are 
two of the three 3-regular graphs of girth 7 \cite{mmwebsite}.
The Coxeter graph on 28 vertices is also 
has girth 7 and is $C_6$-saturated.  

Using a $C_6$-builder with 11 vertices and 
exactly two copies of $C_4$, we 
can prove the following upper bound 
on $\textup{sat}(n , C_4 , C_6)$.

\begin{theorem}\label{c4 against c6}
If $t \geq 1$ is an integer, then 
\[
\textup{sat}(10t + 1 , C_4 , C_6) \leq 2t.
\]
\end{theorem}
\begin{proof}
The graph 
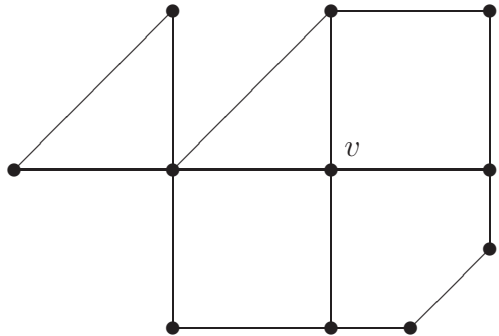
\begin{center}
\begin{picture}(180,120)
\put(0,60){\circle*{5}}
\put(60,0){\circle*{5}}
\put(60,60){\circle*{5}}
\put(60,120){\circle*{5}}
\put(120,0){\circle*{5}}
\put(120,60){\circle*{5}}
\put(120,120){\circle*{5}}
\put(150,0){\circle*{5}}
\put(180,30){\circle*{5}}
\put(180,60){\circle*{5}}
\put(180,120){\circle*{5}}

\put(125,65){$v$}

\put(0,60){\line(1,0){60}}
\put(0,60){\line(1,1){60}}
\put(60,0){\line(0,1){60}}
\put(60,0){\line(1,0){60}}
\put(60,60){\line(1,1){60}}
\put(60,60){\line(0,1){60}}
\put(60,60){\line(1,0){60}}
\put(120,0){\line(0,1){60}}
\put(120,0){\line(1,0){30}}
\put(120,60){\line(0,1){60}}
\put(120,60){\line(1,0){60}}
\put(120,120){\line(1,0){60}}
\put(150,0){\line(1,1){30}}
\put(180,30){\line(0,1){30}}
\put(180,60){\line(0,1){60}}
\end{picture}
\end{center}
is a $C_6$-builder 
on 11 vertices with exactly 2 copies of 
$C_4$.  
The vertex $v$ is a distinguished vertex.
If $t$ copies of this builder are 
glued together at $v$,  
then we obtain a $C_6$-saturated graph 
on $10t +1$ vertices with $2t$ triangles.  The 
computations that show this graph 
is a $C_6$-builder (and has 2 copies of $C_4$) may be 
found in \cite{mathematica}.
\end{proof}

 \section{Open Problems}\label{section problems}
 We end with some open problems. First, when $H$ and $F$ were both cliques we were not able to determine the function $\sat(n, H, F)$ except when counting triangles in a $K_4$-saturated graph. We believe that the natural construction giving the upper bound in Theorem \ref{small improvement} is correct.
 
 \begin{problem}
 Let $s>r\geq 3$ be integers. Determine the exact value of $\mathrm{sat}(n, K_r, K_s)$.
 \end{problem}
 
One of the most intriguing questions for us was counting triangles in $C_4$-saturated graphs. In Section \ref{k3s in c4-sat} we showed that $\limsup_{n\to\infty} \mathrm{sat}(n, K_3, C_4) \leq \frac{n-1}{2}$, but we could not show that $\liminf_{n\to \infty} \textup{sat}(n,K_3,C_4)$ is positive.

\begin{problem}
Determine if $\textup{sat}(n , K_3 , C_4)$ is positive for infinitely many $n$.  
\end{problem}

We ask the same question when counting copies of $C_4$ in a $C_6$-saturated graph.

\begin{problem}
Determine if $\textup{sat}(n , C_4 , C_6)$ is positive for infinitely many $n$.  
\end{problem}

Finally, we focused on graphs which are either $C_k$-saturated or $K_s$-saturated. It would be interesting to consider other nontrivial combinations of graphs $H$ and $F$, for example when one of them is a tree.

\subsection*{Acknowledgement}
The second author is grateful to Beka Ergemlidze for many helpful discussions.

 \end{document}